\documentclass[10pt]{article}
\usepackage{amsmath, amssymb, amsthm, graphicx,float,mathrsfs,enumitem,xpatch,thmtools,wasysym,algorithm,algorithmic}
\usepackage{hyperref,cleveref}
\usepackage[margin=1in]{geometry}
\usepackage{setspace}
\usepackage{units}[loose,nice]
\usepackage{multirow}
\usepackage{booktabs}

\usepackage[toc]{appendix}

\usepackage{xcolor,soul}

\newcommand{\set}[1]{\{#1\}}

\newcommand{\bigo}{\mathcal{O}}
\newcommand{\real}{\mathbb{R}}

\DeclareMathOperator{\argmin}{argmin}

\title{Newton-Anderson at Singular Points}
\author{Matt Dallas\\ Department of Mathematics, University of Florida \\ Sara Pollock\\Department of Mathematics, University of Florida}
  \newtheorem{assumption}{Assumption}[section]
  \newtheorem{remark}{Remark}[section]
  \newtheorem{definition}{Definition}[section]
  \newtheorem{theorem}{Theorem}[section]
  \newtheorem{lemma}{Lemma}[section]
  \newtheorem{proposition}{Proposition}[section]

\begin{document}

\maketitle

\section{Introduction}
\label{sec:introduction}

Given a nonlinear function $f:\mathbb{R}^n\to\mathbb{R}^n$ and root $x^*$ for which 
$f(x^*)=0$, it is well known that if the derivative of $f$ at the root is nonsingular, 
then Newton's method exhibits quadratic convergence in a sufficiently small ball 
centered at the root. On the other hand, if the derivative is singular at $x^*$, e.g., 
at a bifurcation point \cite{Se79}, Newton's method converges linearly in a star-like 
region containing the root \cite{Gr80}. 
This singular setting has been studied in great detail 
\cite{DeKeKe83,DeKe80,Gr80,Re78,Re79}, 
and a number of acceleration schemes have been proposed and analyzed 
\cite{BeFi12,BelMo15,DeKeKe83,DeKe82,HMT09,KaYaFu04,KeSu83,ScFr84}. For example, the Levenberg-Marquardt method featured in \cite{BeFi12,BelMo15,KaYaFu04} is known to be effective for solving nonlinear systems with singular Jacobians under the local-error bound condition.
The focus of this paper is the analysis 
and demonstration of an extrapolation scheme called Anderson 
acceleration, sometimes called Anderson mixing, e.g.,\cite{TRL04}, applied to 
Newton's method, for singular problems.
We will show that with the proposed safeguarding 
strategy, the method is both 
theoretically sound and can be beneficial in practice for singular problems.

Anderson acceleration was first proposed in 
\cite{Anderson65}, in the context of integral equations, to improve the convergence 
of fixed-point iterations.
Anderson acceleration is an attractive method to improve the convergence of linearly
converging fixed-point iterations due to its low computational cost, ease of 
implementation, and track record of success over a wide range of problems. 
The method recombines the $m$ most recent update steps and iterates to form an 
accelerated iterate at each stage of a given fixed-point method, where the particular
combination is generally given as the solution to a least-squares problem.  
Here $m$ may be
referred to as the algorithmic depth, which is often chosen small, say less than five
\cite{ToKe15}, but may sometimes benefit from being taken substantially larger
\cite{PoRe21,WaNi11}.
The method has been found beneficial in diverse applications, such as 
the computation of canonical tensor decompositions \cite{WHdS21},
the study of block copolymer systems \cite{TRL04}, 
geometry optimization and simulation \cite{sim-app-2018}, 
flow problems \cite{LWWY12,PRX18}, and electronic-structure computations 
\cite{AJW17,FaSa09},
to name a few. 
Substantial advances in understanding the method in relation to 
generalized Broyden methods and (nonlinear) GMRES 
are developed in \cite{Eyert96,FaSa09,WaNi11}.
Recently, significant effort has been devoted to analyzing Anderson acceleration applied 
to contractive and noncontractive operators with certain nondegeneracy assumptions 
\cite{EPRX19,PRX18,PoRe21,ToKe15}. 

Here, we will focus on the analysis of Anderson acceleration applied to 
Newton's method for a problem of the form $f(x) = 0$, when the derivative $f'$ 
is singular at a root $x^*$. 
Rapid convergence of the accelerated scheme in comparison with standard Newton 
has been demonstrated numerically in this singular case \cite{posc20}, where it is also 
observed that it is generally both sufficient and preferable
to set the algorithmic depth to $m=1$. It was also found in \cite{EPRX19} in 
a nondegenerate setting that Anderson accelerated Newton iterations with algorithmic
depth $m=1$ could converge where Newton iterations failed, but that increasing $m$
only slowed convergence.
In the remainder, we will
consider Anderson acceleration with depth $m=1$ applied to Newton iterations, which we 
will refer to simply as Newton-Anderson. 
In comparison to the accelerated Newton methods of 
\cite{DeKeKe83,DeKe82,KeSu83}, Newton-Anderson may be seen as advantageous
as it does not require explicit knowledge of the order of the root 
(defined in section \ref{sec:high_order_roots}), or
construction of an approximate projection mapping onto the null space of $f'(x)$.
In further contrast to these predictor-corrector methods, Newton-Anderson requires a 
single linear solve per iteration.
An analysis of Newton-Anderson in the one-dimensional singular case is presented in 
\cite{Po21_book}; 
however, to our knowledge no previous convergence theory has been developed for 
dimension $n > 1$.
The goal of this paper is to provide such a theory.
 
The remainder of the paper is organized as follows.
The underlying foundation of the analysis relies on
a technique for approximating the inverse of the derivative near a given point as developed in \cite{DeKeKe83}. We discuss this technique in \cref{preliminaries}, and in \cref{err-exp} apply it to a Newton-Anderson step to obtain an expansion of the error at step $k$. 
We then analyze this expansion in 
sections \ref{null-comp-and-theta} and 
\ref{analysis-of-pair-types} with a one-step analysis of the error based on previous
consecutive error-pairs, revealing the mechanism behind the changes in convergence rate
demonstrated by the method.
The main challenge of proving convergence for any Newton-like method in the singular case is that the geometry of the region of invertibility is more complex. To handle this, in \cref{convergence} we introduce a novel safeguarding scheme, which we call $\gamma$-safeguarding. This technique leads to the main results of this paper: 
when the null space 
of the derivative at the root is one-dimensional, then under the same conditions 
implying local convergence of the standard Newton method e.g., \cite[Theorem 1.2]{DeKe80}, Newton-Anderson with $\gamma$-safeguarding exhibits local convergence, and in general the rate of convergence is improved. We extend these results to high order roots in section \ref{sec:high_order_roots}. 
 These results are then demonstrated numerically in 
\cref{numerical-results} with several standard benchmark problems, both singular and nonsingular, including the Chandrasekhar H-equation \cite{chandra60,Ke18}. The introduced $\gamma$-Newton-Anderson is further shown to perform
favorably in comparison to existing methods developed for the problem class both in 
terms of robustness and efficiency.

\section{Preliminaries}
\label{preliminaries}
Let $f:\mathbb{R}^n\to\mathbb{R}^n$ be a $C^3$ function such that $f(x^*)=0$ with $x^*\in \mathbb{R}^n$. 
This regularity assumption is standard for the problem class; see, for
example \cite{DeKeKe83,DeKe80,DeKe82,GrOs83,Re78,Re79}.
Suppose $N=\text{null }\big(f'(x^*)\big)$ is nontrivial,  let $R=\text{range }\big(f'(x^*)\big)$, and let 
$\mathbb{R}^n=N\oplus R$. Throughout this paper, $B_r(x)$ denotes a ball of radius $r>0$ 
centered at $x$, $P_N$ and $P_R$ denote the orthogonal projections onto $N$ and $R$ respectively.
Denote the error by  $e_k=x_k-x^*$, and the Newton update step by 
$w_{k+1}=-f'(x_k)^{-1}f(x_k)$.
We take the {\em singular set} $S$ to be the set of all $x\in\mathbb{R}^n$ such that $\det \big(f'(x)\big)=0$.
When $\|\cdot\|=\|\cdot\|_2$, the \emph{Newton-Anderson algorithm} reads as follows.

\begin{algorithm}[H]
\begin{algorithmic}
\caption{Newton-Anderson}
\label{alg:n.anderson}
\STATE{Choose $x_0\in\mathbb{R}^n$. Set $w_1=-f'(x_0)^{-1}f(x_0)$, and $x_1=x_0+w_1$.}
\FOR{k=1,2,...}
    \STATE $w_{k+1}\gets -f'(x_k)^{-1}f(x_k)$
    \STATE $\gamma_{k+1}\gets (w_{k+1}-w_k)^Tw_{k+1}/\|w_{k+1}-w_k\|_2^2$
    \STATE $x_{k+1}\gets x_k+w_{k+1}-\gamma_{k+1}(x_k-x_{k-1}+w_{k+1}-w_k)$
\ENDFOR 
\end{algorithmic}
\end{algorithm}

Let $\hat{D}_N(x)(\cdot):=P_Nf''(x^*)(P_N(x-x^*),P_N(\cdot))$. Here we're writing $f''(x^*)$ as the bilinear map $f''(x^*)(\cdot,\cdot)$ on $\real^n\times\real^n$. Hence $\hat{D}_N(x)(\cdot)$ is a linear map from $N$ to $N$. 
If $\hat{D}_N(x)$ is invertible as a map on $N$ whenever $P_N(x-x^*)\neq 0$, one can show that there exists constants $\hat{\rho}>0$ and $\hat{\sigma}>0$ for which $f'(x)$ is invertible in the region
\begin{align}
    \hat{W}:=W(\hat{\rho},\hat{\sigma},x^*)=\set{x\in\mathbb{R}^n : \, \|x-x^*\|<\hat{\rho}  ,\,\|P_R(x-x^*)\|<\hat{\sigma}\|P_N(x-x^*)\|},
\end{align} 
and that the standard Newton iterates remain in this region for sufficiently small $\|e_0\|$. Further, $f'(x)^{-1}=\hat{D}_N(x)^{-1}+\bigo(1)$ and $\|f'(x)^{-1}\|\leq c\|x-x^*\|^{-1}$. See \cite{DeKeKe83,DeKe80}, or \cite{Gr80} for details. 
 When accelerating Newton's method, care must be taken to ensure the accelerated iterates remain in $W(\rho,\sigma,x^*)$. We won't insist that the iterates lie in $W(\rho,\sigma,x^*)$ in \cref{err-exp} through \cref{analysis-of-pair-types}, as the primary focus is a one-step analysis. Rather, we'll make the more relaxed assumption
that 
$\hat{D}(x_i)(\cdot)=P_Nf''(x^*)(e_i,P_N(\cdot))$ 
is  invertible as a map on $N$ in $B_r(x^*)\setminus S$, and the error $\|e_i\|$ is so small that $f'(x_i)$ is invertible so that $f'(x_i)^{-1}=\hat{D}(x_i)^{-1}+\bigo(1)$ holds with $i=k,k-1$. 
 Note that near $N$, $\hat{D}(x_i)=P_Nf''(x^*)(e_i,P_N)$ is invertible if and only if $\hat{D}_N(x_i)$ is invertible. In \cref{convergence}, we'll show that with safeguarding and  sufficiently small $\|e_0\|$, the Newton-Anderson iterates remain in $\hat{W}$ if $x_0\in \hat{W}$. 
We remark that $\hat{D}_N(x)$ fails to be invertible for all $P_N(x-x^*)\neq 0$ if $\dim N>1$ and odd (see \cite[p.148]{GrOs81}). In these cases, one may instead assume that there exists a $\varphi\in N$ for which the linear map $\hat{D}(\varphi)=P_Nf''(x^*)(\varphi,P_N)$ is invertible, and then work with the set
 \begin{align}
     W(\rho,\sigma,\xi,x^*)=W \cap \set{x\in\mathbb{R}^n : \|(P_N-P_{\varphi})(x-x^*)\|<\xi\|P_{\varphi}(x-x^*)\|}. 
 \end{align}
 Here $P_{\varphi}$ denotes the projection onto the one-dimensional subspace of $N$ spanned by $\varphi$ and $W=W(\rho,\sigma,x^*)$. This approach may be found in 
\cite{DeKeKe83,KeSu83,Re79}. A more general analysis of regions of invertibility may be found in \cite{Gr80}. For our purposes studying rates of convergence, the sets $W(\rho,\sigma,x^*)$ and $W(\rho,\sigma,\xi,x^*)$ suffice. In \cref{convergence} we focus on the case when $\dim N=1$, where we are able to work with the larger set  $W(\rho,\sigma,x^*)$. For reference, we explicitly state
 \begin{assumption}\label{order-of-root-assumption}
 The linear operator $\hat{D}(x)$ is invertible as a map on $N$ in the ball $B_{\hat{r}}(x^*)\setminus S$, and $\|x-x^*\|<\hat{r}$ implies $f'(x)^{-1}=\hat{D}(x)^{-1}+\bigo(1)$.
 \end{assumption}

Next we consider a standard error expansion for the analysis of Newton's
method for singular problems, and show how the behavior of Newton-Anderson iterates
differs from the behavior of Newton iterates without acceleration.
\section{Error Expansion}
\label{err-exp}
\begin{theorem}\label{error-expansion}
Let \cref{order-of-root-assumption} hold, and let $x_k\in B_{\hat{r}}(x^*)\setminus S$. 
Then 
\begin{align}
    \label{initial-keybound1}
    e_k+w_{k+1}&= \frac{1}{2}P_Ne_k+\frac{1}{2}\hat{D}^{-1}(x_k)f''(x_k)(e_k,P_Re_k)+\bigo(\|e_k\|^2)\\ 
    \label{initial-keybound2}
    w_{k+1}&=-\frac{1}{2}P_Ne_k+\bigg(\frac{1}{2}\hat{D}^{-1}(x_k)f''(x_k)(e_k,\cdot)-I\bigg)P_Re_k+\bigo(\|e_k\|^2).
\end{align}
\end{theorem}

Bounds similar to \eqref{initial-keybound1} may be found in \cite{DeKeKe83,DeKe80,Re79}. Note that \eqref{initial-keybound2} follows from \eqref{initial-keybound1} by subtracting $e_k$ to the right hand side, writing $e_k=P_Ne_k+P_Re_k$, and grouping like terms.

Henceforth, we will let $T_k(\cdot)=(1/2)\hat{D}^{-1}(x_k)f''(x_k)(e_k,\cdot)$. Let $\gamma_{k+1}$ be the coefficient computed in Algorithm \ref{alg:n.anderson}.  
For $x,y\in\mathbb{R}^n$, define the function 
\begin{align}\label{newt-and-sum}
    L_{k+1}(x,y):=(1-\gamma_{k+1})x+\gamma_{k+1}y.
\end{align} 
We will call any term of the form $L_{k+1}(x,y)$ the \textit{Newton-Anderson sum} of $x$ and $y$. When the inputs are indexed, such as $x_k$ and $x_{k-1}$, we write $L_{k+1}(x_k,x_{k-1})=x_k^{\alpha}$. As in \cite{PoRe21}, we define the \textit{optimization gain} 
\begin{align}\label{optimization-gain-def}
    \theta_{k+1}:=\|w_{k+1}^{\alpha}\|/\|w_{k+1}\|=\|w_{k+1}-\gamma_{k+1}(w_{k+1}-w_k)\|/\|w_{k+1}\|.
\end{align} 
In \cite{PoRe21}, $\theta_{k+1}$ was shown to be the key quantity determining the acceleration from a Newton-Anderson step in the nonsingular case. Here, in the singular case, it will be shown to be the key quantity determining the acceleration of the error along the null component. It was also shown in \cite{PoRe21} that $\theta_{k+1}=|\sin(w_{k+1}-w_k,w_k)|$, where $\sin(w_{k+1}-w_k,w_k)^2=1-\big((w_{k+1}-w_k)^Tw_{k+1}\big)^2/(\|w_{k+1}-w_k\|^2\|w_{k+1}\|^2)$ is the direction sine between $w_{k+1}-w_k$ and $w_{k+1}$. Therefore, $\theta_{k+1}$ is small when $w_{k+1}-w_k$ is nearly parallel to $w_{k+1}$. 

With the notation described in the preceding paragraph, \eqref{initial-keybound1} and \eqref{initial-keybound2} become 
\begin{align}
    \label{keybound1}
    e_k+w_{k+1}&= \frac{1}{2}P_Ne_k+T_kP_Re_k+\bigo(\|e_k\|^2)\\ 
    \label{keybound2}
    w_{k+1}&=-\frac{1}{2}P_Ne_k+\big(T_k-I\big)P_Re_k+\bigo(\|e_k\|^2).
\end{align}

The following proposition provides the error expansions for a Newton-Anderson step that will be fundamental to our analysis. 
\begin{proposition}
    Let $q_{k-1}^k$ denote a term for which $\|q_{k-1}^k\|\leq c\big(|1-\gamma_{k+1}|\,\|e_k\|^2+|\gamma_{k+1}|\,\|e_{k-1}\|^2\big)$. Under the assumptions of theorem \eqref{error-expansion}, if $x_{k+1}$ is the $(k+1)$-st Newton-Anderson iterate, then we can expand $e_{k+1}$ and $w_{k+1}^{\alpha}$ as  
    \begin{align}\label{newton-anderson-error}
    e_{k+1}&=\frac{1}{2}P_Ne_k^{\alpha}+\left(T_kP_Re_k\right)^{\alpha}+q_{k-1}^k\\ \label{theta-w-expansion}
    w_{k+1}^{\alpha}&=-\frac{1}{2}P_Ne_k^{\alpha}+\left((T_k-I)P_Re_k\right)^{\alpha}+q_{k-1}^k.
    \end{align}
\end{proposition}

\begin{proof}
Given $x_k$ and $x_{k-1}$, a Newton-Anderson step takes the form $x_{k+1} = L_{k+1}(x_k+w_{k+1},x_{k-1}+w_k)=(1-\gamma_{k+1})(x_k+w_{k+1})+\gamma_{k+1}(x_{k-1}+w_k)$. It follows that 
\begin{align}
	e_{k+1} = (1-\gamma_{k+1})(e_k+w_{k+1})+\gamma_{k+1}(e_{k-1}+w_k).
\end{align}
Applying \eqref{keybound1} to $e_k+w_{k+1}$ and $e_{k-1}+w_k$, and grouping up appropriate terms yields 
\begin{align}
	e_{k+1} &= \frac{(1-\gamma_{k+1})}{2}P_Ne_k+\frac{\gamma_{k+1}}{2}P_Ne_{k-1}+(1-\gamma_{k+1})T_kP_Re_k \\
&+\gamma_{k+1}T_{k-1}P_Re_{k-1} 
+(1-\gamma_{k+1})\bigo(\|e_k\|^2)+\gamma_{k+1}\bigo(\|e_{k-1}\|^2).\nonumber
\end{align}
Writing $q_{k-1}^k = (1-\gamma_{k+1})\bigo(\|e_k\|^2)+\gamma_{k+1}\bigo(\|e_{k-1}\|^2)$ and using the $\alpha$ notation for a Newton-Anderson sum, we have 
\begin{align}
	e_{k+1} = \frac{1}{2} P_Ne_k^{\alpha}+\left(T_kP_Re_k\right)^{\alpha}+q_{k-1}^k.
\end{align}
By analogous reasoning, applying \eqref{keybound2} to $(1-\gamma_{k+1})w_{k+1}+\gamma_{k+1}w_k$ yields \eqref{theta-w-expansion}.
\end{proof}

The structure of $f$ leads to a simple upper bound on $P_Re_{k+1}$. Apply $P_R$ to \cref{newton-anderson-error}. Since the range of $T_k$ lies in $N$, the only term remaining on the right hand side is $q_{k-1}^k$. Thus  
    \begin{align}\label{range-comp-bound}
         \|P_Re_{k+1}\|\leq |1-\gamma_{k+1}|C_1\|e_k\|^2+|\gamma_{k+1}|C_2\|e_{k-1}\|^2.
    \end{align}
The constants $C_1$ and $C_2$ are independent of $k$ and depend on $f$.
The bound in \eqref{range-comp-bound} resembles the result of Lemma 1 in \cite{posc20} for depth $m=1$. There, the Jacobian is assumed to be nonsingular at the solution $x^*$, 
and the $\gamma$ coefficients are assumed bounded. We do not assume $\gamma_{k+1}$ is bounded for our one-step analysis in sections \ref{null-comp-and-theta} and \ref{analysis-of-pair-types}. When we consider convergence in section \ref{convergence}, $\gamma$-safeguarding will ensure that the $\gamma_{k+1}$'s remain bounded in the region of convergence.
The point here is that like a standard Newton step, the range component of the error from a Newton-Anderson step behaves as if the Jacobian were nonsingular at the solution. Thus, when the Jacobian is singular at $x^*$, the source of slow convergence must come from the null space component. Conversely, if Newton-Anderson is seen to accelerate a given Newton sequence, and $f'(x^*)$ is singular, then it must accelerate the null component. If $P_Re_{k+1}$ is the only accelerated component, then the source of the linear convergence remains unaltered, and linear convergence would still be observed. Therefore much of the analysis focuses on the null component error, the foundation of which is the notion of \textit{pair-types}.
This method introduces a new technique to the one-step analysis of 
Anderson acceleration, and explains how the convergence rate changes at different 
steps for Newton-Anderson.

\subsection{Pair Types}
\label{pair-types}

We may consider the Newton-Anderson algorithm as acting on ordered pairs $(x_k,x_{k-1})$, where $x_k$ and $x_{k-1}$ are the previous two Newton-Anderson iterates. We then analyze the output, $x_{k+1}$, based on where the vectors $x_k$ and $x_{k-1}$ lie in $\mathbb{R}^n$, e.g., $x_k$ and $x_{k-1}$ both lie near $N$, or $x_k$ lies near $N$ and $x_{k-1}$ lies near $R$. This leads to the notion of pair types, which we now define. Here, ``dominant" means greatest in norm relative to the other terms on the right hand side of \eqref{keybound1}. We'll make this more explicit in the next section.

\begin{definition}\label{pair-type-def}
Let $\set{x_k}$ be a sequence of Newton-Anderson iterates. 
\begin{enumerate}
    \item $(x_k,x_{k-1})$ is an \emph{N-pair} if $(1/2)P_Ne_i$ is the dominant term in \eqref{keybound1} for $i=k,k-1$,
    \item $(x_k,x_{k-1})$ is an \emph{R-pair} if $T_iP_Re_i$ is the dominant term in \eqref{keybound1} for $i=k,k-1$, 
    \item $(x_k,x_{k-1})$ is an \emph{NR-pair} if $(1/2)P_Ne_i$ is the dominant term for $i=k$, and $T_iP_Re_i$ is the dominant term for $i=k-1$, and 
    \item $(x_k,x_{k-1})$ is an \emph{RN-pair} if $T_iP_Re_i$ is the dominant term for $i=k$, and $(1/2)P_Ne_i$ is the dominant term for $i=k-1$. 
\end{enumerate}

Each respective pair is a \emph{strong pair} if the corresponding Newton-Anderson sum is dominant in equation \eqref{newton-anderson-error}.
\end{definition}

Strong pairs are those for which the dominant terms in $x_k$ and $x_{k-1}$ remain dominant in the Newton-Anderson step.
Our analysis focuses on strong pairs, and of particular interest is how Newton-Anderson acts on strong N-pairs, since near $N$ the standard Newton method exhibits linear convergence.
One of the most interesting, but perhaps not surprising, results is that strong N-pairs are the most responsive to a successful optimization step in the Newton-Anderson algorithm. In other words, the pairs that are ``close" to the null space, the ``slow" region for standard Newton, stand to gain the most from Anderson acceleration.

\Cref{pair-flow-chart-update} below summarizes the results of the technical analysis in
sections \ref{null-comp-and-theta} and \ref{analysis-of-pair-types}. The left most path is optimal in the sense that we have effectively recovered the bound from the nonsingular case. The right most path is the most Newton-like case, where there is little or no acceleration of the null components. The middle path, in which $(x_k,x_{k-1})$ is compatible (see definition \ref{def:compatible}), lies between these two cases. As $\theta_{k+1}\to 1$, $x_{k+1}$ tends to look more like a standard Newton step, and as $\theta_{k+1}\to 0$, we see acceleration  of $x_{k+1}$. In this sense, $\theta_{k+1}$ interpolates between the nonsingular case where we have superlinear order, and the singular case with linear order. Note that \cref{loading-step-prop} can be interpreted as saying that no more than two consecutive pairs can follow the right most path in \cref{pair-flow-chart-update}. 

\begin{figure}[H]
    \centering
    \includegraphics[width=120mm]{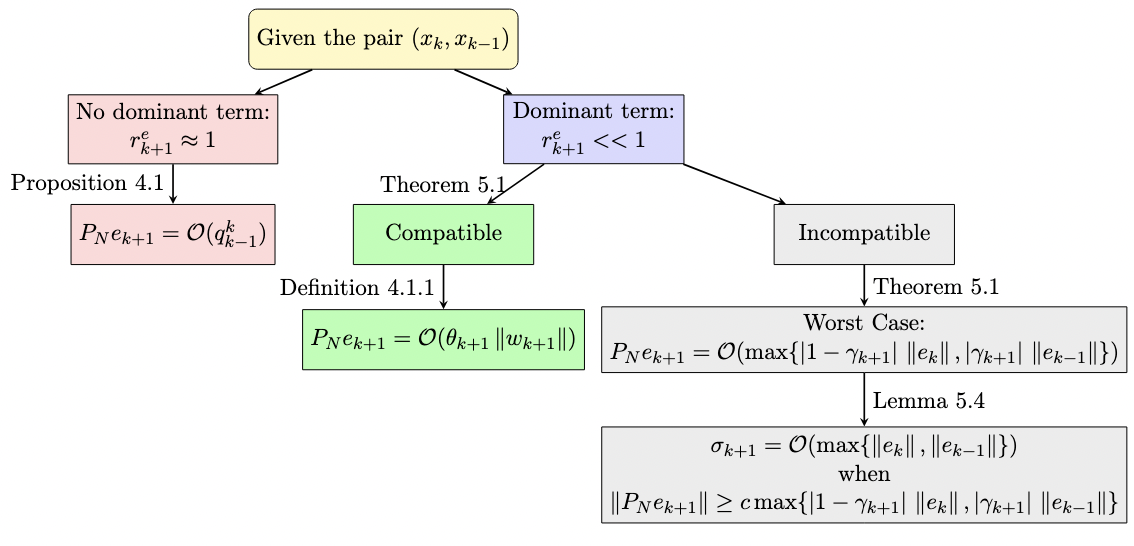}
    \caption{Summary of results from sections \cref{null-comp-and-theta} and \cref{analysis-of-pair-types}, with labels noting where each relation is proven or defined. }
    \label{pair-flow-chart-update}
\end{figure}

\section{The Null Space Component and Compatibility}\label{null-comp-and-theta}

Our goal in this section is to derive a general relation between the null component error and the optimization gain $\theta_{k+1}$. Such a relation will allow us to rigorously describe how the optimization step in \cref{alg:n.anderson} accelerates the null component. This relation is given in proposition \eqref{theta-controls-error-prop}. The strategy is to take the largest term (with respect to the norm) in the expansion of $P_Ne_{k+1}$ from \eqref{newton-anderson-error} and find conditions under which this term is bounded by $\theta_{k+1}\|w_{k+1}\|$. We'll then see (in the next section) when each strong pair-type satisfies these conditions. First, 
consider each term in expansion  \eqref{newton-anderson-error}.  Namely, 
\begin{align}\label{listed-terms-pn}
   \frac{(1-\gamma_{k+1})P_Ne_k}{2},\hspace{1mm} \frac{\gamma_{k+1}P_Ne_{k-1}}{2},\hspace{1mm}(1-\gamma_{k+1})T_kP_Re_k,\hspace{2mm}\gamma_{k+1}T_{k-1}P_Re_{k-1}, \text{ and }P_Nq_{k-1}^k.
\end{align}
Any sum of a given strict subset of these terms could be the dominant term in equation \eqref{newton-anderson-error}. 
We'll let $S^e_{k+1}$ denote the set of all such sums. 
For example, if $(x_k,x_{k-1})$ is a strong N-pair as defined in Definition \eqref{pair-type-def}, then $(1-\gamma_{k+1})P_Ne_k+\gamma_{k+1}P_Ne_{k-1}$ is the dominant term. 
We'll write 
\begin{align}\label{set-of-ratios-e}
R_{k+1}^e := \{ \|P_Ne_{k+1}-\overset{\sim}{E}_{k+1}\|/\|\overset{\sim}{E}_{k+1}\| : \overset{\sim}{E}_{k+1}\in S^e_{k+1} \},
\end{align}
and 
$E_{k+1} := \argmin R_{k+1}^e.$
We then define 
\begin{align}\label{min-ratio-def}
r_{k+1}^e:= \min R_{k+1}^e = \|P_Ne_{k+1}-E_{k+1}\|/\|E_{k+1}\|.
\end{align} 

Observe that if $q\in R_{k+1}^e$, then $1/q\in R_{k+1}^e$. This gives the following lemma.

\begin{lemma}\label{min-ratio-is-less-than-one} The minimum ratio $r_{k+1}^e$ satisfies $r_{k+1}^e\leq 1$ for all $k\geq 0$. 
\end{lemma}

Analogous notation will be used for the expansion $P_N w_{k+1}^{\alpha}$ in \eqref{theta-w-expansion}. 
Namely, $S^{w}_{k+1}$ will denote the set of sums of strict subsets of the terms
\begin{align}\label{listed-terms-wn}
  -\frac{(1-\gamma_{k+1})}{2}P_Ne_k,\hspace{0.5mm} -\frac{\gamma_{k+1}P_Ne_{k-1}}{2},\hspace{0.5mm}(1-\gamma_{k+1})T_{k}P_Re_k,\hspace{0.5mm}\gamma_{k+1}T_{k-1}P_Re_{k-1}, \text{ and }P_Nq_{k-1}^k,
\end{align}

We'll write $R_{k+1}^w = \{\|P_Nw_{k+1}^{\alpha}-\overset{\sim}{F}_{k+1}\|/\|\overset{\sim}{F}_{k+1}\| : \overset{\sim}{F}_{k+1}\in S^w_{k+1}\}$,
$F_{k+1} := \argmin R_{k+1}^w$, 
and 
$r_{k+1}^w = \min R_{k+1}^w.$

We now state and prove the proposition referenced at the beginning of this section
that provides a bound on the null space component of the error in terms of the optimization gain $\theta_{k+1}$ (defined in \eqref{optimization-gain-def}).

\begin{proposition}
\label{theta-controls-error-prop} {Let the assumptions of} \cref{error-expansion} hold 
 for $x_k$ and $x_{k-1}$. Let $q_{k-1}^k$ denote a term for which $\|q_{k-1}^k\|\leq c\big(|1-\gamma_{k+1}|\,\|e_k\|^2+|\gamma_{k+1}|\,\|e_{k-1}\|^2\big)$, and suppose
    \begin{align}\label{compatibility-condition-general}
        \|E_{k+1}\|\leq \frac{\|P_Nw_{k+1}^{\alpha}\|}{1-r_{k+1}^e}.
    \end{align}
Then 
\begin{align}\label{null-bound-theta}
    \|P_Ne_{k+1}\|\leq 
\bigg(\frac{1+r_{k+1}^e}{1-r_{k+1}^e}\bigg)\theta_{k+1}\|w_{k+1}\|. 
\end{align}

\end{proposition}

\begin{proof}
Adding and subtracting $E_{k+1}$ to $P_Ne_{k+1}$ gives $\|P_Ne_{k+1}\|\leq (1+r_{k+1}^e)\|E_{k+1}\|\leq (1+r_{k+1}^e)(1-r_{k+1}^e)^{-1}\|P_Nw_{k+1}^{\alpha}\|$.
Then observing that
\begin{align}
     \bigg(\frac{1+r_{k+1}^e}{1-r_{k+1}^e}\bigg)\|P_Nw_{k+1}^{\alpha}\| 
    \leq \bigg(\frac{1+r_{k+1}^e}{1-r_{k+1}^e}\bigg)\|w_{k+1}^{\alpha}\|
    =\bigg(\frac{1+r_{k+1}^e}{1-r_{k+1}^e}\bigg)\theta_{k+1}\|w_{k+1}\|\nonumber
\end{align}
completes the proof.
\end{proof}

We'll soon state definition \ref{def:compatible}, which is motivated by proposition \ref{theta-controls-error-prop}, but first we'll prove two related propositions. The following gives sufficient conditions for \eqref{compatibility-condition-general} to hold.

\begin{proposition}\label{compatbility-criteria-remark} 
The relation \eqref{compatibility-condition-general} holds if
$r_{k+1}^e < 1$, and there is an element $\overset{\sim}{F}_{k+1}\in S_{k+1}^w$ 
such that $\|\overset{\sim}{F}_{k+1}\|=\|E_{k+1}\|$, and $\|P_Ne_{k+1}-E_{k+1}\|=\|P_Nw_{k+1}^{\alpha}-\overset{\sim}{F}_{k+1}\|$. 
\end{proposition}
\begin{proof}
Suppose $\|\overset{\sim}{F}_{k+1}\|=\|E_{k+1}\|$ and $\|P_Ne_{k+1}-E_{k+1}\|=\|P_Nw_{k+1}^{\alpha}-\overset{\sim}{F}_{k+1}\|$. Then $r_{k+1}^e=\|P_Nw_{k+1}^{\alpha}-\overset{\sim}{F}_{k+1}\|/\|\overset{\sim}{F}_{k+1}\|$, and $r_{k+1}^e\leq 1$ by \cref{min-ratio-is-less-than-one}. If $r_{k+1}^e<1$, then 
\begin{align}
    \|P_Nw_{k+1}^{\alpha}\|&\geq \|\overset{\sim}{F}_{k+1}\|-\|P_Nw_{k+1}^{\alpha}-\overset{\sim}{F}_{k+1}\|
    =(1-r_{k+1}^e)\|E_{k+1}\|.
\end{align}
\end{proof}

As $r_{k+1}^e$ approaches 1, the denominator in \eqref{null-bound-theta} approaches zero, resulting in a poor bound. However, the following proposition shows that  $r_{k+1}^e\approx 1$ implies $P_Ne_{k+1}$ was accelerated. 
\begin{proposition}\label{quadratic-case}
Let $r_{k+1}^e$ be defined as in \eqref{min-ratio-def} and let the assumptions of
\cref{error-expansion} hold for $x_k$ and $x_{k-1}$. If $0\leq \varepsilon< 1$ and $r_{k+1}^e=1-\varepsilon$, then $P_Ne_{k+1}=\big((2-\varepsilon)/(1-\varepsilon)\big)\bigo(\max\set{\|e_k\|^2,\|e_{k-1}\|^2})$. 
\end{proposition}
\begin{proof}
Since $r_{k+1}^e\leq \|P_Ne_{k+1}-\overset{\sim}{E}_{k+1}\|/\|\overset{\sim}{E}_{k+1}\|$ for any $\overset{\sim}{E}_{k+1}\in S_{k+1}^e$,
we have
\begin{align}
    1-\varepsilon=r_{k+1}^e\leq \frac{\|P_Ne_{k+1}-\overset{\sim}{E}_{k+1}\|}{\|\overset{\sim}{E}_{k+1}\|}
\end{align}
for all $\overset{\sim}{E}_{k+1}\in S_{k+1}^e$. Taking $\overset{\sim}{E}_{k+1}=(1/2)P_Ne_k^{\alpha}+T_kP_Re_k^{\alpha}$, it follows that 
\begin{align}
    1-\varepsilon\leq \frac{\|q_{k-1}^k\|}{\|(1/2)P_Ne_k^{\alpha}+T_kP_Re_k^{\alpha}\|}. 
\end{align}
Thus $(1-\varepsilon)\|(1/2)P_Ne_k^{\alpha}+T_kP_Re_k^{\alpha}\|\leq \|q_{k-1}^k\|$, and it follows that $\|P_Ne_{k+1}\|\leq \big((2-\varepsilon)/(1-\varepsilon)\big)\|q_{k-1}^k\|$.
\end{proof}

Now we state definition \ref{def:compatible}.

\begin{definition}\label{def:compatible}
Let $\set{x_k}$ be a sequence of Newton-Anderson iterates. We say that $x_{k+1}$ is \emph{compatible} or a \emph{compatible step} if there exists a moderate constant $C>0$ independent of $k$ such that $\|P_Ne_{k+1}\|\leq C \theta_{k+1}\|w_{k+1}\|$, in which case we'll write $P_Ne_{k+1}=\bigo(\theta_{k+1}\|w_{k+1}\|)$. Otherwise,
    $(x_k,x_{k-1})$ is an \emph{incompatible pair}, and $x_{k+1}$ is \emph{incompatible} or an \emph{incompatible step} 
\end{definition}
We note that this particular use of $\bigo$ is common, e.g., \cite{Wr95}.\\

\textit{Compatible} here is suggestive of the result of \cref{theta-controls-error-prop}. If $(x_k,x_{k-1})$ is a pair that satisfies \cref{theta-controls-error-prop}, then $P_Ne_{k+1}=\bigo(\theta_{k+1}\|w_{k+1}\|)$. Hence a successful optimization step in Newton-Anderson implies acceleration of $P_Ne_{k+1}$, and therefore of $e_{k+1}$. This can be seen by applying \eqref{keybound2} to bound $w_{k+1}$ in terms of $e_k$. In particular, when $(x_k,x_{k-1})$ is a strong N-pair, \eqref{keybound2} implies that $\|w_{k+1}\|\leq (1/2)(1+c_1\sigma_k+c_2(1+\sigma_k)\|e_k\|)\,\|P_Ne_k\|$, where $\sigma_k = \|P_Re_k\|/\|P_Ne_k\|$, and $c_1$ and $c_2$ are constants determined by $f$. Combining this with compatibility gives 
\begin{align}\label{theta-acceleration}
	\|P_Ne_{k+1}\|\leq C\theta_{k+1}(1/2)(1+c_1\sigma_k+c_2(1+\sigma_k)\|e_k\|)\,\|P_Ne_k\|.
\end{align}
The error for the null component of an analogous standard Newton step, i.e., one where $(1/2)\|P_Ne_k\|$ is norm-dominant on the right-hand-side of \eqref{keybound1}, is given by 
\begin{align}\label{no-theta-acceleration}
\|P_N(e_k+w_{k+1})\|\leq (1/2)(1+c_1\sigma_k+c_2(1+\sigma_k)\|e_k\|)\|P_Ne_k\|.	
\end{align}
In $\hat{W}$, the region were the Jacobian is invertible defined in section \ref{sec:introduction}, we can bound $\sigma_k\leq \hat{\sigma}$ and $\|e_k\|\leq \hat{\rho}$, and we can expect $C$ in \eqref{theta-acceleration} to be moderate in size. In appendix \ref{appendix-b}, with $\gamma$-safeguarding, we obtain a bound of the form \eqref{theta-acceleration} with $C<1$ for sufficiently small $\hat{\sigma}$ and $\hat{\rho}$. Hence the bound in \eqref{theta-acceleration} for a Newton-Anderson step in $\hat{W}$ is essentially the bound seen in \eqref{no-theta-acceleration} for a standard Newton step in $\hat{W}$ scaled by $\theta_{k+1}$. Since $\theta_{k+1}\leq 1$, this implies that in $\hat{W}$, a Newton-Anderson step will be no worse than a Newton step, and in the case of a successfull optimization step, i.e., when $\theta_{k+1}$ is small, we have acceleration. Compatible steps also nicely mirror the Anderson theory developed in \cite{PoRe21} under certain nondegeneracy assumptions, where it's shown that small $\theta_{k+1}$ results in acceleration, and $\theta_{k+1}\approx 1$ results in a standard, non-accelerated step. In this sense, compatible pairs are those that behave like nonsingular pairs, where the mechanism behind the acceleration is $\theta_{k+1}$.
When $(x_k,x_{k-1})$ is \textit{incompatible}, i.e., $P_Ne_{k+1}\neq \bigo(\theta_{k+1}\|w_{k+1}\|)$, acceleration may still occur. However, this is not guaranteed, and a successful optimization step does not imply acceleration of $P_Ne_{k+1}$.
So far, it has been shown that the expansion of $P_Ne_{k+1}$ in \eqref{newton-anderson-error} either has a dominant term, and this term is  $\bigo(\theta_{k+1}\|w_{k+1}\|)$ when $(x_k,x_{k-1})$ is a compatible pair, or there is no dominant term and $P_Ne_{k+1}=\bigo(\max\{||e_k||^2,||e_{k-1}||^2\})$.

\section{Analysis of Pair Types}
\label{analysis-of-pair-types}

In this section, the results of \cref{null-comp-and-theta} are applied to pair types. As stated in the paragraph following \cref{pair-type-def}, we focus on strong pair types. If a given pair $(x_k,x_{k-1})$ is not strong, and there is a dominant term in equation \eqref{newton-anderson-error}, then this term can be analyzed analogously to the strong terms analyzed here. 
Moreover, strong pairs, in particular strong N-pairs, are most relevant for the convergence theory developed in section \ref{convergence}, and theorem \ref{loading-step-prop} proven later in this section essentially says that after two consecutive steps with little improvement in the null component of the error, the next step will either be a strong N-pair or be bounded only by higher order terms. 

\subsection{Compatibility Conditions}
Our aim here is to establish conditions for each strong pair type under which compatibility is assured. That is, conditions under which $P_Ne_{k+1}=\bigo(\theta_{k+1}\|w_{k+1}\|)$. Evidently, each pair type can be compatible if certain alignment conditions are met. These vary by pair type, with strong mixed pairs having the most stringent alignment conditions. On the other hand, strong N-pairs are automatically compatible.
\begin{lemma}\label{strong-n-pairs-are-compatible}
{Let the assumptions in} \cref{error-expansion}  hold for $x_k$ and $x_{k-1}$, and suppose $r_{k+1}^e < 1$.
If $(x_k,x_{k-1})$ is a strong N-pair, then $(x_k,x_{k-1})$ is a compatible pair. 
\end{lemma}
\begin{proof}
Let $E_{k+1}$ and $F_{k+1}$ be defined as in the discussion preceding \cref{theta-controls-error-prop}. To show that a strong N-pair is compatible, it suffices to prove \eqref{compatibility-condition-general}. Here, the condition from \cref{compatbility-criteria-remark} is used. Suppose $(x_k,x_{k-1})$ is a strong N-pair, so that $E_{k+1}=(1/2)P_Ne_k^{\alpha}$. Let $F_{k+1}=-(1/2)P_Ne_k^{\alpha}$. Then $\|E_{k+1}\|=\|F_{k+1}\|$, and 
\begin{align}
\|P_Nw_{k+1}^{\alpha}-F_{k+1}\|=\|T_kP_Re_k^{\alpha}+P_Nq_{k-1}^k\|=\|P_Ne_{k+1}-E_{k+1}\|.
\end{align} 
Relation \eqref{compatibility-condition-general} then follows by \cref{compatbility-criteria-remark}, by which $\|P_Ne_{k+1}\|=\bigo(\theta_{k+1}\|w_{k+1}\|)$. Thus $(x_k,x_{k-1})$ is a compatible pair. 
\end{proof}

\begin{remark} Lemma \ref{strong-n-pairs-are-compatible} is significant because it says that 
if $x_k$ and $x_{k-1}$ are close to $N$, and the resulting Newton-Anderson step remains 
near $N$, the region in which standard Newton is slowest, then $(x_k,x_{k-1})$ is 
guaranteed to be compatible, and therefore $P_Ne_{k+1}$ is controlled by 
$\theta_{k+1}\|w_{k+1}\|$ which implies acceleration when $\theta_{k+1}$ is small as described in the paragraph following definition \ref{def:compatible}.
This justifies the statement at the end of 
\cref{err-exp} that pairs near $N$ are the most responsive to a successful 
optimization step in the Newton-Anderson algorithm. In \cref{convergence} we introduce $\gamma$-safeguarding, which will ensure the iterates remain near $N$ if $x_0$ is chosen near $N$,
thereby guaranteeing the iterates remain well-defined and compatible.
\end{remark}

Lemma \ref{strong-r-pair-bounds}  gives conditions in which a strong R-pair is compatible, and a bound for the incompatible case. Note that in the case of an incompatible strong R-pair, we obtain a quadratic bound. However, we have no guarantee that the new iterate generated from this strong R-pair will be well-defined, 
as the Jacobian may not be invertible at this iterate.

\begin{lemma}\label{strong-r-pair-bounds}
{Let the assumptions in} \cref{error-expansion} hold for $x_k$ and $x_{k-1}$, 
let $(x_k,x_{k-1})$ be a strong R-pair, and suppose $r_{k+1}^e<1$. Then $(x_k,x_{k-1})$ is compatible if $(P_Ne_k^{\alpha})^Tq_{k-1}^k\geq 0$. Otherwise, if $(x_k,x_{k-1})$ is incompatible, then 
\begin{align}\label{strong-r-pair-incompatible-upper-bound}
    \|P_Ne_{k+1}\|\leq (1+r_{k+1}^e)\bigg(|1-\gamma_{k+1}|\,C_1\,q_{k-2}^{k-1}+|\gamma_{k+1}|\,C_2\,q_{k-3}^{k-2}
    \bigg),
\end{align}
where $q_{i-1}^i$ is a term such that $\|q_{i-1}^i\|\leq |1-\gamma_{i+1}|C_3\|e_i\|^2+C_4\bigo(\|e_{i-1}\|^2)$ for $i\in\set{k-2,k-1}$, where $C_3$ and $C_4$ are constants determined by $f$.
\end{lemma}

\begin{proof}
First, it's shown that $(P_Ne_k^{\alpha})^Tq_{k-1}^k\geq 0$ implies $(x_k,x_{k-1})$ is compatible. Since $(x_k,x_{k-1})$ is strong R-pair, the dominant term $E_{k+1}$ is $T_kP_Re_k^{\alpha}$. By \cref{theta-controls-error-prop}, it suffices to show 
\begin{align}
    \|T_kP_Re_k^{\alpha}\|\leq \frac{\|P_Nw_{k+1}^{\alpha}\|}{1-r_{k+1}^e},
\end{align}
where $r_{k+1}^e$ is defined as in \eqref{min-ratio-def}. Here $r_{k+1}^e=\|(1/2)P_Ne_k^{\alpha}+q_{k-1}^k\|/\|T_kP_Re_k^{\alpha}\|$.

For any vectors $u$ and $v$ in $\mathbb{R}^n$, the polarization identity $4u^Tv=\|u+v\|^2-\|u-v\|^2$ implies that if $u^Tv\geq 0$, then $\|u-v\|\leq \|u+v\|$. Letting $u=q_{k-1}^k$ and $v=(1/2)P_Ne_k^{\alpha}$, it follows that if $\left(P_Ne_k^{\alpha}\right)^Tq_{k-1}^k\geq 0$, then $\|q_{k-1}^k-(1/2)P_Ne_k^{\alpha}\|\leq \|q_{k-1}^k+(1/2)P_Ne_k^{\alpha}\|$. By \eqref{theta-w-expansion}, 
\begin{align}
    \|T_kP_Re_k^{\alpha}\|\leq \frac{\|P_Nw_{k+1}^{\alpha}\|}{1-\frac{\|q_{k-1}^k-(1/2)P_Ne_k^{\alpha}\|}{\|T_kP_Re_k^{\alpha}\|}}\leq \frac{\|P_Nw_{k+1}^{\alpha}\|}{1-\frac{\|q_{k-1}^k+(1/2)P_Ne_k^{\alpha}\|}{\|T_kP_Re_k^{\alpha}\|}}.
\end{align}
The last inequality follows because $\|q_{k-1}^k-(1/2)P_Ne_k^{\alpha}\|\leq \|q_{k-1}^k+(1/2)P_Ne_k^{\alpha}\|$. Since $r_{k+1}^e=\|(1/2)P_Ne_k^{\alpha}+q_{k-1}^k\|/\|T_kP_Re_k^{\alpha}\|$, $(x_k,x_{k-1})$ is compatible by \cref{theta-controls-error-prop}. When $(x_k,x_{k-1})$ is incompatible, we still have $\|P_Ne_{k+1}\|\leq (1+r_{k+1}^e)\|T_kP_Re_k^{\alpha}\|$. After applying the triangle inequality and \eqref{range-comp-bound} to $\|T_kP_Re_k^{\alpha}\|$ we arrive at 
\begin{align}
    \|T_kP_Re_k^{\alpha}\|\leq |1-\gamma_{k+1}|\,C_1\,q_{k-2}^{k-1}+|\gamma_{k+1}|\,C_2\,q_{k-3}^{k-2}.
\end{align}
This gives the bound in \eqref{strong-r-pair-incompatible-upper-bound}.
\end{proof}

Now we come to strong mixed pairs, the case with the most stringent alignment conditions. As in \cref{theta-controls-error-prop}, $E_{k+1}=\argmin R_{k+1}^e$, where $R_{k+1}^e$ is defined in \eqref{set-of-ratios-e}.
We will also use the \textit{Newton-Anderson sum} notation $L_k(x,y)$ defined in \eqref{newt-and-sum}. 
\begin{lemma}\label{strong-mixed-compatibility-cond}
{Let the assumptions of} \cref{error-expansion} hold for $x_k$ and $x_{k-1}$, 
and suppose $r_{k+1}^e < 1$. Suppose $(x_k,x_{k-1})$ is a strong mixed pair. There are two cases. 
    \begin{enumerate}
        \item (Strong NR-pair) If $E_{k+1}=L_{k+1}(P_Ne_k,T_{k-1}P_Re_{k-1})$, then $(x_k,x_{k-1})$ is compatible if 
        \begin{enumerate}
            \item $(\,(1-\gamma_{k+1})P_Ne_k)^T(\gamma_{k+1}T_{k-1}P_Re_{k-1})\leq 0$, and \item $(\gamma_{k+1}P_Ne_{k-1})^T((1-\gamma_{k+1})T_kP_Re_k+q_{k-1}^k)\geq 0$.
        \end{enumerate}   
        
When $(x_k,x_{k-1})$ is not compatible, 
        \begin{align}
        \scalebox{0.93}{
            $\|P_Ne_{k+1}\|\leq C(1+r_{k+1}^e)\max\set{|1-\gamma_{k+1}|\,\|P_Ne_{k}\|,|\gamma_{k+1}|\,\|T_{k-1}P_Re_{k-1}\|}.$}
        \end{align}
        \item (Strong RN-pair) If $E_{k+1}=L_{k+1}(T_kP_Re_k,P_Ne_{k-1})$, then $(x_k,x_{k-1})$ is compatible if 
        \begin{enumerate}
            \item $(\,(1-\gamma_{k+1})T_kP_Re_k)^T(\gamma_{k+1}P_Ne_{k-1})\leq 0$, and \item $((1-\gamma_{k+1}P_Ne_{k})^T(\gamma_{k+1}T_{k-1}P_Re_{k-1}+q_{k-1}^k)\geq 0$.
        \end{enumerate}   When $(x_k,x_{k-1})$ is not compatible, 
        \begin{align}
        \scalebox{0.90}{
            $\|P_Ne_{k+1}\|\leq C(1+r_{k+1}^e)\max\set{|1-\gamma_{k+1}|\,\|T_kP_Re_k\|,(|\gamma_{k+1}|/2)\,\|P_Ne_{k-1}\|}.$}
        \end{align}
    \end{enumerate}
    In both cases, $C$ denotes a constant determined by $f$. 
\end{lemma}

\begin{proof}
The proof is very similar to that of  \cref{strong-r-pair-bounds}. We only discuss the part that uses the additional alignment conditions, namely conditions (a) in the lemma statement. Further, the proofs are identical for strong NR-pairs and strong RN-pairs. We'll focus on strong NR-pairs. By \cref{theta-controls-error-prop}, it suffices to show $\|E_{k+1}\|\leq \|P_Nw_{k+1}^{\alpha}\|/(1-r_{k+1}^e)$. Suppose $(x_k,x_{k-1})$ is a strong NR-pair. Then $E_{k+1}=(1-\gamma_{k+1})P_Ne_k/2+\gamma_{k+1}T_{k-1}P_Re_{k-1}$. 
Write \eqref{theta-w-expansion} as 
\begin{align}
P_Nw_{k+1}^{\alpha} &= -\frac{1}{2}(1-\gamma_{k+1})P_Ne_k+\gamma_{k+1}T_{k-1}P_Re_{k-1} \\ 
&-\frac{1}{2}\gamma_{k+1}P_Ne_{k-1}+(1-\gamma_{k+1})T_kP_Re_k+P_Nq_{k-1}^k.\nonumber
\end{align}
Applying the reverse triangle inequality gives 
\begin{align}
\|P_Nw_{k+1}^{\alpha}\| &\geq \left\|-\frac{1}{2}(1-\gamma_{k+1})P_Ne_k+\gamma_{k+1}T_{k-1}P_Re_{k-1}\right\|  \\ 
	&- \left\| -\frac{1}{2}\gamma_{k+1}P_Ne_{k-1}+(1-\gamma_{k+1})T_kP_Re_k+P_Nq_{k-1}^k\right\|.\nonumber
\end{align}
Factoring $ \|-\frac{1}{2}(1-\gamma_{k+1})P_Ne_k+\gamma_{k+1}T_{k-1}P_Re_{k-1}\|$ and dividing gives 
\begin{align}
    \|\gamma_{k+1}T_{k-1}P_Re_{k-1}-\frac{(1-\gamma_{k+1})}{2}P_Ne_k\|\leq \frac{\|P_Nw_{k+1}^{\alpha}\|}{1-\frac{\|(1-\gamma_{k+1})T_kP_Re_k-(\gamma_{k+1})/2)P_Ne_{k-1}+q_{k-1}^k\|}{\|\gamma_{k+1}T_{k-1}P_Re_{k-1}-((1-\gamma_{k+1})/2)P_Ne_k\|}}.
\end{align}
Since $(\,(1-\gamma_{k+1})P_Ne_k)^T(\gamma_{k+1}T_{k-1}P_Re_{k-1})\leq 0$, the polarization identity implies that $\|\gamma_{k+1}T_{k-1}P_Re_{k-1}+((1-\gamma_{k+1})/2)P_Ne_k\|\leq \|\gamma_{k+1}T_{k-1}P_Re_{k-1}-((1-\gamma_{k+1})/2)P_Ne_k\|$. Hence
\begin{align}\label{penultimate-step}
    \|((1-\gamma_{k+1})/2)P_Ne_k+\gamma_{k+1}T_{k-1}P_Re_{k-1}\|\leq \frac{\|P_Nw_{k+1}^{\alpha}\|}{1-\frac{\|(1-\gamma_{k+1})T_kP_Re_k-(\gamma_{k+1}/2)P_Ne_{k-1}+q_{k-1}^k\|}{\|\gamma_{k+1}T_{k-1}P_Re_{k-1}+(-(1-\gamma_{k+1})/2)P_Ne_k\|}}.
\end{align}
Now apply $(\gamma_{k+1}P_Ne_{k-1})^T((1-\gamma_{k+1})T_kP_Re_k+q_{k-1}^k)\geq 0$ to the right hand side of \eqref{penultimate-step}, and proceed as in the proof of \cref{strong-r-pair-bounds}. The proof of the strong RN-pair case is identical.
\end{proof}
\begin{remark}
In principal , as long as $f'(x_0)$ and $f'(x_1)$ are invertible, $x_0$ and $x_1$ can be chosen arbitrarily, but this strategy may lead to a mixed pair $(x_0,x_1)$. The stringent alignment conditions on strong mixed pairs suggest the better strategy is to chose $x_0$ near $N$ and take $x_1=x_0+w_1$ as written in \cref{alg:gamma-safeguarding-pseudocode}. This way both $x_1$ and $x_0$ will be near $N$, and in \cref{convergence} this will be shown to imply compatibility when $\gamma$-safeguarding is used. 
\end{remark}
The results of lemmas \ref{strong-n-pairs-are-compatible}-\ref{strong-mixed-compatibility-cond} together are summarized in the following theorem.

\begin{theorem}\label{summary-of-compatibility}
{Let the assumptions of} \cref{error-expansion} hold for $x_k$ and $x_{k-1}$. 
Suppose $(x_k,x_{k-1})$ is a 
strong pair. If $(x_k,x_{k-1})$ is a strong N-pair, then $(x_k,x_{k-1})$ is compatible. 
Otherwise, there are alignment conditions under which $(x_k,x_{k-1})$ is compatible. 
If $(x_k,x_{k-1})$ is an incompatible strong R-pair, then $P_Ne_{k+1}$ is bounded by higher order terms. 
If $(x_k,x_{k-1})$ is an incompatible strong mixed pair, then $P_Ne_{k+1}$ is no worse 
than $\bigo\big(\max\{|1-\gamma_{k+1}|\,\|e_k\|,|\gamma_{k+1}|\,\|e_{k-1}\|\}\big)$.  
\end{theorem}
    
    \subsection{Incompatible Pairs}
    In this section, we investigate the incompatible case further. We'll use $\sigma_{k+1}:=\|P_Re_{k+1}\|/\|P_Ne_{k+1}\|$ introduced in \cref{preliminaries}. One can think of $\sigma_{k+1}$ as measuring the angle between $e_{k+1}$ and the null space $N$, but it can also be interpreted as measuring the acceleration of $\|P_Ne_{k+1}\|$ relative to the standard Newton algorithm.
In a standard Newton step, $P_Ne_{k+1}=\bigo(\|e_k\|)$, and $P_Re_{k+1}=\bigo(\|e_k\|^2)$. Hence $\sigma_{k+1}=\bigo(\|e_k\|)$. So we can think of any Newton-Anderson step as ``Newton-like" when $\sigma_{k+1}=\bigo(\max\set{|1-\gamma_{k+1}|\,\|e_k\|,|\gamma_{k+1}|\,\|e_{k-1}\|})$, and as accelerated, if $\sigma_{k+1}=\bigo(\max\set{|1-\gamma_{k+1}|\,\|e_k\|^{\ell},|\gamma_{k+1}|\,\|e_{k-1}\|^{\ell}})$ with $0\leq \ell<1$.  From this perspective, the worst-case for a Newton-Anderson step is when $P_Ne_{k+1}$ is {incompatible and} bounded below by first order in the sense that $\|E_{k+1}\|>C\max\set{|1-\gamma_{k+1}|\,\|e_k\|,\,|\gamma_{k+1}|\,\|e_{k-1}\|}$. This can occur if there is a single term in \eqref{listed-terms-pn} that is much larger than the others in norm, or if $(x_k,x_{k-1})$ is an incompatible strong mixed pair since $P_Ne_{k+1}=\bigo(|1-\gamma_{k+1}|\,\|e_k\|,|\gamma_{k+1}|\,\|e_{k-1}\|)$ in the worst case by \cref{strong-mixed-compatibility-cond}.

\begin{lemma} \label{newton-like-sigma-bound}
Let the assumptions of \cref{error-expansion} hold for $x_k$ and $x_{k-1}$.
As in \cref{theta-controls-error-prop}, let $E_{k+1}=\argmin R_{k+1}^e$.
If $\|E_{k+1}\|>C\max\set{|1-\gamma_{k+1}|\,\|e_k\|,\,|\gamma_{k+1}|\,\|e_{k-1}\|}$, where $C>0$ is some constant dependent on
$f$, and $1-r_{k+1}^e\neq 0$, then  $\sigma_{k+1}=\bigo(\max\set{\|e_k\|,\|e_{k-1}\|})$. 
\end{lemma}

\newpage
\begin{proof}
By \eqref{range-comp-bound}, $\|P_Re_{k+1}\|\leq C\max\set{|1-\gamma_{k+1}|\,\|e_k\|^2,|\gamma_{k+1}|\,\|e_{k-1}\|^2}$. Let

$M_{k,k-1} = \max\set{|1-\gamma_{k+1}|\,\|e_k\|,\,|\gamma_{k+1}|\,\|e_{k-1}\|}$.
The assumption on $E_{k+1}$ allows us to bound $\|P_Ne_{k+1}\|$ below:
\begin{align}
\|P_Ne_{k+1}\| &= \|P_Ne_{k+1}-E_{k+1}+E_{k+1}\|  \\ 
&\geq (1-r_{k+1}^e)\|E_{k+1}\|\nonumber \\
&\geq C(1-r_{k+1}^e)M_{k,k-1}\nonumber.
\end{align} 

Combining the constants from the upper bound on $\|P_Re_{k+1}\|$ and the lower bound on $\|P_Ne_{k+1}\|$, we have 
\begin{align}
    \sigma_{k+1}=\frac{\|P_Re_{k+1}\|}{\|P_Ne_{k+1}\|}&\leq \frac{C}{1-r_{k+1}^e}\frac{\max\set{|1-\gamma_{k+1}|\,\|e_k\|^2,|\gamma_{k+1}|\,\|e_{k-1}\|^2}}{\max\set{|1-\gamma_{k+1}|\,\|e_k\|, |\gamma_{k+1}|\,\|e_{k-1}\|}} \\
    &\leq \frac{C}{1-r_{k+1}^e}\max\set{\|e_k\|,\|e_{k-1}\|}.\nonumber
\end{align}
The last inequality holds since $\max\set{|1-\gamma_{k+1}|\,\|e_k\|,|\gamma_{k+1}|\,\|e_{k-1}\|}\geq |1-\gamma_{k+1}|\,\|e_k\|$ and $|\gamma_{k+1}|\,\|e_{k-1}\|$.
\end{proof}

\begin{remark}\label{remark-on-sigma-bound}
Under the same assumptions, we can make the slightly stronger statement: Suppose $\|E_{k+1}\|> C\max\set{|1-\gamma_{k+1}|\,\|e_k\|^{\ell},|\gamma_{k+1}|\,\|e_{k-1}\|^{\ell}}$, where $\ell\in\mathbb{R}$ and $0<\ell<2$, then 

$\sigma_{k+1}\leq(C/(1-r_{k+1}^e))\max\set{\|e_k\|^{2-\ell},\|e_{k-1}\|^{2-\ell}}$. In this paper, however, we'll only use the case $\ell=1$ discussed in lemma \ref{newton-like-sigma-bound}.
\end{remark}

 By definition, $\|P_Re_{k+1}\|=\sigma_{k+1}\|P_Ne_{k+1}\|$. So if $\sigma_{k+1}=\bigo(\max\set{\|e_k\|,\|e_{k-1}\|})$, then 
\begin{align}
    \label{range-comp-bounded-by-sigma}
    \|P_Re_{k+1}\|\leq C\max\set{\|e_k\|,\|e_{k-1}\|}\|P_Ne_{k+1}\|.
\end{align}
This says that if $x_{k+1}$ is a Newton-like step and 
$\max\set{\|e_k\|,\|e_{k-1}\|}\ll 1$, then $\|P_Re_{k+1}\|<<\|P_Ne_{k+1}\|$. Therefore $x_{k+1}$ is close to $N$. That is, Newton-like steps send iterates towards $N$, the region in which compatibility is guaranteed
by \cref{strong-n-pairs-are-compatible}. Moreover, the smaller $\max\set{\|e_k\|,\|e_{k-1}\|}$ is, the closer Newton-like steps are to $N$. One may think that after a finite number of these Newton-like steps, the iterates will be sufficiently clustered around $N$ to yield a compatible pair. This is indeed the case, as is shown in the following theorem.  

\begin{theorem}\label{loading-step-prop}
Let \cref{order-of-root-assumption} hold, and let $x_i\in B_{\hat{r}}(x^*)\setminus S$ for $i=k,k-1,k-2,$ and $k-3$. 
Let $E_k = \argmin R_{k}^e$ and $E_{k-1}=\argmin R_{k-1}^e$.
Suppose $x_j$ and $x_{j-1}$ are incompatible with  $\|E_{j}\|>C\max\set{|1-\gamma_{j}|\,\|e_{j-1}\|,\,|\gamma_{j}|\,\|e_{j-2}\|}$, and $r_{j}^e$ is bounded away from one for $j=k,k-1$. Then $P_Ne_{k+1}$ is either $\bigo(\theta_{k+1}\|w_{k+1}\|)$, i.e., $(x_k,x_{k-1})$ is compatible, or consists only of higher order terms.
\end{theorem}

\begin{proof}
Apply the triangle inequality to \eqref{newton-anderson-error} to get 
\begin{align}\label{superlinear-or-compatible}
    \|P_Ne_{k+1}\|&\leq \frac{1}{2}\|P_Ne_k^{\alpha}\|+|1-\gamma_{k+1}|\,C\,\|P_Re_k\|\\ 
&+|\gamma_{k+1}|\,C\,\|P_Re_{k-1}\|+\|q_{k-1}^k\|.\nonumber
\end{align}
By \cref{newton-like-sigma-bound} and \eqref{range-comp-bounded-by-sigma}, $\|P_Re_i\|\leq C\|e_i\|\max\set{\|e_{i-1}\|,\|e_{i-2}\|}$ for $i=k,k-1$.
Grouping the higher order terms into a single term denoted $Q_{k+1}$,\footnote{That is, $Q_{k+1}$ denotes a term such that $\|Q_{k+1}\|\leq |1-\gamma_{k+1}|\,C\,\|e_k\|\max\set{\|e_{k-1}\|,\|e_{k-2}\|}+|\gamma_{k+1}|\,C\,\|e_{k-1}\|\max\set{\|e_{k-2}\|,\|e_{k-3}\|}+q_{k-1}^k$.} \eqref{superlinear-or-compatible} becomes 
\begin{align}\label{superlinear-or-compatible-clean}
    \|P_Ne_{k+1}\|\leq \frac{1}{2}\|P_Ne_k^{\alpha}\|+Q_{k+1}.
\end{align}
There are now two cases. If $(1/2)\|P_Ne_k^{\alpha}\|>\|Q_{k+1}\|$, then 

\begin{align}
    \|P_Ne_{k+1}\|&\leq \bigg((1+\frac{\|Q_{k+1}\|}{(1/2)\|P_Ne_k^{\alpha}\|}\bigg)(1/2)\|P_Ne_k^{\alpha}\|
 \\   
& \leq\bigg(\frac{1+\frac{\|Q_{k+1}\|}{(1/2)\|P_Ne_k^{\alpha}\|}}{1-\frac{\|Q_{k+1}\|}{(1/2)\|P_Ne_k^{\alpha}\|}}\bigg)\theta_{k+1}\|w_{k+1}\|.\nonumber
\end{align}
The last inequality follows from applying the reverse triangle inequality to the expansion of $P_Nw_{k+1}^{\alpha}$ from \eqref{theta-w-expansion} as in the proof of \cref{compatbility-criteria-remark}. Otherwise, if $(1/2)\|P_Ne_k^{\alpha}\|\leq \|Q_{k+1}\|$. Then $\|P_Ne_{k+1}\|\leq 2\|Q_{k+1}\|$, so that $P_Ne_{k+1}$ is bounded only by higher order terms.
\end{proof}
From \cref{strong-mixed-compatibility-cond}, we know that incompatible pairs do not necessarily lead to improved error. However, we are assured by \cref{loading-step-prop} that no more than two steps with little or no decrease in error may occur, at which point we can expect improved error (locally) either by compatibility, for sufficiently small $\theta_{k+1}$, or from higher order terms dominating. Indeed, if we have two consecutive steps where there is little decrease (if any at all) in the null component of the error, then we can bound $\|E_j\|$ as seen in the statement of theorem \ref{loading-step-prop}. Then as in the proof we'd find that $\|P_Ne_{k+1}\|$ is compatible, implying acceleration for small $\theta_{k+1}$, or that it's bounded only by higher order terms, which also means acceleration near the solution $x^*$. 

In the next section we introduce the safeguarding strategy called $\gamma$-safeguarding, 
which leads to our main result and proof of convergence, \cref{thm:convergence-thm}.

\section{Convergence}
\label{convergence}
Here we restrict our attention to the case when $\dim N=1$. 
Recall from \cref{preliminaries} that if $\hat{D}_N(x)$ is invertible as a map on $N$ whenever $P_N(x-x^*)\neq 0$, then $f'(x)$ is invertible for all $x\in \hat{W}$. 
We'll show that with an appropriate safeguarding scheme, which we call \emph{$\gamma$-safeguarding}, Newton-Anderson iterates remain in $\hat{W}$ if $x_0\in \hat{W}$, and converge locally under the same conditions that imply local convergence of the standard Newton method. This safeguarding scheme provides an automated way to decide how to scale $\gamma_{k+1}$ so that $\nu_{k+1}$ in equation \ref{eta-notation} below remains bounded away from one. This will prove useful in the proof of theorem \ref{thm:convergence-thm}. Another interpretation is that by scaling $\gamma_{k+1}$ towards zero when appropriate, and therefore taking a more ``Newton-like" step, we prevent the null-space component from accelerating too much and possibly leaving the region of invertibility. We now prove \cref{gamma-safeguarding}, which provides theoretical justification 
for $\gamma$-safeguarding.
\begin{lemma}\label{gamma-safeguarding}
For $\lambda\in (0,1]$, let
    \begin{align}\label{eta-notation}
    \nu_{k+1}=\frac{\min\set{|1-\lambda\gamma_{k+1}|\,\|P_N(e_k+w_{k+1})\|,|\lambda\gamma_{k+1}|\,\|P_N(e_{k-1}+w_k)\|}}{\max\set{|1-\lambda\gamma_{k+1}|\,\|P_N(e_k+w_{k+1})\|,|\lambda\gamma_{k+1}|\,\|P_N(e_{k-1}+w_k)\|}}.
\end{align}
Fix $0<r<1$. Assume $|\gamma_{k+1}|<1$ and nonzero. Then given $\sigma_i<\hat{\sigma}$ and $\rho_i<\hat{\rho}$, $i=k,k-1$, there exists a number $\lambda\in(0,1]$ such that $\nu_{k+1}\leq r<1$ and $\min\set{|1-\lambda\gamma_{k+1}|\,\|P_N(e_k+w_{k+1}),|\lambda\gamma_{k+1}|\,\|P_N(e_{k-1}+w_k)\|}=|\lambda\gamma_{k+1}|\,\|P_N(e_{k-1}+w_k)\|$. 
\end{lemma}

\begin{proof}
The key relation is 
\begin{align}\label{nu-apprx}
  \frac{|\gamma_{k+1}|\,\|P_N(e_{k-1}+w_k)\|}{|1-\gamma_{k+1}|\,\|P_N(e_k+w_{k+1})\|}\leq \frac{\bigg(\frac{1/2+c_3\sigma_{k}+c_2(1+\sigma_{k})\|e_{k}\|}{1/2-c_1\sigma_{k}-c_2(1+\sigma_{k})\|e_{k}\|}\bigg)}{\bigg(\frac{1/2-c_3\sigma_{k-1}-c_2(1+\sigma_{k-1})\|e_{k-1}\|}{1/2+c_1\sigma_{k-1}+c_2(1+\sigma_{k-1})\|e_{k-1}\|}\bigg)}\frac{|\gamma_{k+1}\,\|w_k\|}{|1-\gamma_{k+1}|\,\|w_{k+1}\|}.
\end{align}

To prove \eqref{nu-apprx}, it suffices to show 
\begin{align}\label{e-by-w}
\frac{\left(1/2-c_1\sigma_k-c_2(1+\sigma_k)\|e_k\|\right)}{\left(1/2+c_3\sigma_k+c_2(1+\sigma_k)\|e_k\|\right)}\|w_{k+1}\|\leq \|P_N(e_k+w_{k+1})\|\leq
\frac{\left(1/2+c_1\sigma_k+c_2(1+\sigma_k)\|e_k\|\right)}{\left(1/2-c_3\sigma_k-c_2(1+\sigma_k)\|e_k\|\right)}\|w_{k+1}\|
\end{align}
where the constants $c_1,c_2,$ and $c_3$ are determined by $f$. From \eqref{keybound1}, it follows that 
\begin{align}
\|P_N(e_{k}+w_{k+1})\|\leq \left(1/2+\|T_k\|\sigma_k+c\frac{\|e_k\|^2}{\|P_Ne_k\|}\right)\|P_Ne_k\|.
\end{align}
We also have $\|e_k\|\leq (1+\sigma_k)\|P_Ne_k\|$ 
Therefore 
\begin{align}
\|P_N(e_{k}+w_{k+1})\|\leq \left(1/2+\|T_k\|\sigma_k+c(1+\sigma_k)\|e_k\|\right)\|P_Ne_k\|.
\end{align}
We'll now use \eqref{keybound2} to bound $\|P_Ne_k\|$ in terms of $\|w_{k+1}\|$. Indeed, applying the reverse triangle inequality to the right hand side of \eqref{keybound2} gives 
\begin{align}
\|P_Ne_k\|\leq \left(1/2-\|T_k-I\|\sigma_k-c(1+\sigma_k)\|e_k\|\right)^{-1}\|w_{k+1}\|. 
\end{align}
Hence 
\begin{align}
\|P_N(e_{k}+w_{k+1})\|&\leq \left(1/2+\|T_k\|\sigma_k+c(1+\sigma_k)\|e_k\|\right)\|P_Ne_k\|\\ 
&\leq \frac{\left(1/2+\|T_k\|\sigma_k+c(1+\sigma_k)\|e_k\|\right)}{\left(1/2-\|T_k-I\|\sigma_k-c(1+\sigma_k)\|e_k\|\right)}\|w_{k+1}\|.\nonumber
\end{align}
Since $f\in C^3$, $\|T_k\|\leq c_1$ indpendent of $k$ in $W(\hat{\rho},\hat{\sigma},x^*)$, and $\|T_k-I\|\leq c_3$. This gives the second inequality in \eqref{e-by-w}. To obtain the first inequality, use \eqref{keybound1} to obtain a lower bound on $\|P_N(e_k+w_{k+1})\|$ in terms of $\|P_Ne_k\|$, then use \eqref{keybound2} to bound $\|P_Ne_k\|$ below in terms of $\|w_{k+1}\|$. 
Now,
since $\sigma_i<\hat{\sigma}$ and $\|e_i\|<\hat{\rho}$, the first ratio on the right-hand-side of \eqref{nu-apprx} is bounded by a constant $c_4<1+\varepsilon$, with $\varepsilon>0$, for sufficiently small $\hat{\sigma}$ and $\hat{\rho}$. Therefore, to prove the lemma it suffices to show that there exists a number $\lambda\in(0,1]$ such that $|\lambda \gamma_{k+1}|/|1-\lambda\gamma_{k+1}|\leq r\|w_{k+1}\|/((1+\varepsilon)\|w_{k}\|)$. Let $\beta_{k+1}=r\|w_{k+1}\|/((1+\varepsilon)\|w_{k}\|)$. Since $|\gamma_{k+1}|<1$, $|1-\lambda\gamma_{k+1}|=1-\lambda\gamma_{k+1}$. There are two cases to consider. If $\gamma_{k+1}>0$, then it suffices to take $\lambda\leq \beta_{k+1}/(\gamma_{k+1}(1+\beta_{k+1}))$. Now suppose $\gamma_{k+1}<0$. This gives $-\lambda\gamma_{k+1}/(1-\lambda_{k+1}\gamma_{k+1})\leq \beta_{k+1}$, and rearranging yields $\lambda\gamma_{k+1}(\beta_{k+1}-1)\leq \beta_{k+1}$. Note that if $\beta_{k+1}\geq 1$, then $\lambda\gamma_{k+1}(\beta_{k+1}-1)\leq \beta_{k+1}$ holds for any positive $\lambda$ since $\beta_{k+1}\geq 0$. In particular, we may take $\lambda=1$. Hence no safeguarding is necessary. If $\beta_{k+1}<1$, then it suffices to take $\lambda\leq \beta_{k+1}/(\gamma_{k+1}(\beta_{k+1}-1))$. In any of these cases, we may take $\lambda\leq 1$. 
\end{proof}

Based on Lemma \ref{gamma-safeguarding}, the safeguarded version of Algorithm
\ref{alg:n.anderson} is summarized as follows.
\begin{algorithm}[H]
\caption{$\gamma$-safeguarding}
\label{alg:gamma-safeguarding-pseudocode}
\begin{algorithmic}[1]
\STATE {Given $x_k$, $x_{k-1}$, $w_{k+1}$, $w_k$, and $\gamma_{k+1}$. Set $r\in (0,1)$ and $\lambda=1$.}
    \STATE {$\beta \gets r\|w_{k+1}\|/\|w_k\|$}
    \IF{$\gamma_{k+1}=0$ \textbf{ or } $\gamma_{k+1}\geq 1$}
    \STATE {$x_{k+1}\gets x_k+w_{k+1}$}
    \ELSE
    \IF{$|\gamma_{k+1}|/|1-\gamma_{k+1}|>\beta$}
        \IF {$\gamma_{k+1}>0\textbf{ and } \beta/(\gamma_{k+1}(1+\beta))<1$} 
        \STATE $\lambda\gets \beta/(\gamma_{k+1}(1+\beta))$
        \ENDIF
        \IF {$\gamma_{k+1}<0\textbf{ and } 0\leq\beta/(\gamma_{k+1}(\beta-1))<1$} 
        \STATE $\lambda\gets \beta/(\gamma_{k+1}(\beta-1))$
        \ENDIF
    \ENDIF 
    \ENDIF
    \STATE $\gamma_{k+1}\gets\lambda\gamma_{k+1}$
    \STATE $x_{k+1}\gets  x_k+w_{k+1}-\gamma_{k+1}(x_k-x_{k-1}+w_{k+1}-w_k)$
\end{algorithmic}
\end{algorithm}

If $\gamma_{k+1}=0$ or $\gamma_{k+1}\geq 1$, we just take a standard Newton step, i.e., $x_{k+1}=x_k+w_{k+1}$. This condition is similar to that of the safeguarding strategy seen in \cite{posc20}. There, the authors take a Newton step if the direction cosine between $w_{k+1}$ and $w_k$ exceeds $0.942$. The condition used in this paper, taking a Newton step when $\gamma_{k+1}\geq 1$, is a stronger condition than $\cos(w_{k+1},w_k)>0.942$ in the sense that if $\cos(w_{k+1},w_k)\geq 0.942$, then $|\gamma_{k+1}|<2$. On the other hand, $\gamma$-safeguarding is weaker in the sense that, depending on $\|w_{k+1}\|/\|w_k\|$, it's possible that $|\gamma_{k+1}|<1$ even if $\cos(w_{k+1},w_k)>0.942$. It is also important to note that bounding $(|\lambda\gamma_{k+1}|\,\|w_k\|)/(|1-\lambda\gamma_{k+1}|\,\|w_{k+1}\|)$ away from one also bounds $(|\lambda\gamma_{k+1}|\,\|P_Ne_{k-1}\|)/(|1-\lambda\gamma_{k+1}|\,\|P_Ne_k\|)$ away from one by \eqref{keybound2}. Indeed, $\eqref{keybound2}$ implies
\begin{align}
\|w_{k+1}\|&\leq (1/2+c_1\sigma_k+c_2(1+\sigma_k)\|e_k\|)\|P_Ne_k\|,\\
   \|w_{k+1}\|&\geq (1/2-c_1\sigma_k-c_2(1+\sigma_k)\|e_k\|)\|P_Ne_k\|. \nonumber
\end{align}
This fact will be used in the proof of \cref{thm:convergence-thm}. Analogous to \eqref{eta-notation}, we'll let 
\begin{align}
   \overset{\sim}{\nu}_{k+1}=\frac{\min\set{|1-\lambda\gamma_{k+1}|\,\|P_Ne_k\|,|\lambda\gamma_{k+1}|\,\|P_Ne_{k-1}\|}}{\max\set{|1-\lambda\gamma_{k+1}|\,\|P_Ne_k\|,|\lambda\gamma_{k+1}|\,\|P_Ne_{k-1}\|}}.
\end{align}

Some notation: let $\lambda_{k+1}$ be the value computed by algorithm \ref{alg:gamma-safeguarding-pseudocode} for the chosen parameter $r$ at step $k+1$, $\theta_{k+1}^{\lambda}=\|w_{k+1}-\lambda_{k+1}\gamma_{k+1}(w_{k+1}-w_k)\|/\|w_{k+1}\|$, and  $L_{k+1}^{\lambda}(x,y)=(1-\lambda_{k+1}\gamma_{k+1})x+\lambda_{k+1}\gamma_{k+1}y$. 
We now state the main convergence result. The proof may be found in appendix \ref{appendix-b}.
\begin{theorem}\label{thm:convergence-thm}
     Let $\dim N=1$, and let $\hat{D}_N(x)$ be invertible as a map on $N$ for all $P_N(x-x^*)\neq 0$. Let $W_k=W(\|e_k\|,\sigma_k,x^*)$. If $x_0$ is chosen so that $\sigma_0<\hat{\sigma}$ and $\|e_0\|<\hat{\rho}$, for sufficiently small $\hat{\sigma}$ and $\hat{\rho}$, $x_1=x_0+w_1$, and $x_{k+1}=L_{k+1}^{\lambda}(x_k+w_{k+1},x_{k-1}+w_k)$ for $k\geq 1$, then $W_{k+1}\subset W_0$ for all $k\geq 0$ and $x_k\to x^*$. That is, $\set{x_k}$ remains well-defined and converges to $x^*$. Furthermore, there exist constants $c_4>0$ and $\kappa\in(1/2,1)$ such that 
        \begin{align}
         \|P_Re_{k+1}\|&\leq
           c_4\max\set{|1-\lambda_{k+1}\gamma_{k+1}|\,\|e_k\|^2,|\lambda_{k+1}\gamma_{k+1}|\,\|e_{k-1}\|^2}\label{range-bd-main}\\
          \|P_Ne_{k+1}\|&< \kappa\theta_{k+1}^{\lambda}\|P_Ne_{k}\| \label{null-rate-of-convergence-main}
     \end{align}
for all $k\geq 1$.
\end{theorem}

\section{High Order Roots}\label{sec:high_order_roots}
In this section, we generalize the results from the previous sections to higher order roots, which we define now.  Let $p$ denote the smallest integer such that $f^{(p+1)}(x^*)(x-x^*)^p\neq 0$. Here $(x-x^*)^p$ means that the first $p$ arguments of the $(p+1)$-linear map $f^{(p+1)}(x^*)$ are $x-x^*$. The \textit{order} of the root $x^*$ is the smallest integer $d$ such that $P_Nf^{(d+1)}(x^*)(x-x^*)^d\neq 0$. We define the order of $x^*$ this way since $d$ determines the rate of convergence of $P_Nx_k$ when $\{x_k\}$ is generated by Newton's method. With this definition, we can say that the analysis thus far has focused on first order roots. Now we're interested in roots of order greater than one, that is, roots $x^*$ such that $d>1$.  The discussion in section \ref{preliminaries} generalizes in a reasonably straightforward way to higher order roots. See \cite{DeKeKe83} for details on invertibility of $f'(x)$ in this context. We'll continue to assume $\dim N=1$, and we also assume the following throughout this section. Recall that $S$ is the set in which $f'(x)$ is singular. 
\begin{assumption}\label{assump:high-order}
The operator $\hat{D}^{d}(x)=(d!)^{-1}P_Nf^{(d+1)}(x^*)(x-x^*)^{d}P_N$ is invertible as a map on $N$ in the ball $B_R(x^*)\setminus S$, and $R$ is sufficiently small to ensure $f'(x)^{-1}$ exists and $f'(x)^{-1}=\bigo(||x-x^*||^{-d})$. 
\end{assumption}

With this assumption, we have lemma \ref{thm:key-bound-ho}. The expansions in \eqref{null-bound-ho} and \eqref{range-bound-ho} are similar to those found in \cite{DeKeKe83}. Expansion \eqref{key-w-bd-ho} follows from adding \eqref{null-bound-ho} to \eqref{range-bound-ho} and then isolating $w_{k+1}$. We've let $\hat{T}_k(\cdot)$ denote a linear map on $\real^n$ whose bounded independent of $x_k$.
\begin{lemma}\label{thm:key-bound-ho}
Let assumption \eqref{assump:high-order} hold and let $x_k\in B_{\hat{r}}(x^*)\setminus S$. Then 
\begin{align}\label{null-bound-ho}
P_N(e_k+w_{k+1})&=\frac{d}{d+1} P_Ne_k+\hat{T}_kP_Re_k+\bigo(\|e_k\|^2)\\
\label{range-bound-ho}
P_R(e_k+w_{k+1})&=\bigo(\|e_k\|^{d+1})\\ \label{key-w-bd-ho}
w_{k+1}&=-\frac{1}{d+1}P_Ne_k+(\hat{T}_k-I)P_Re_k+\bigo(\|e_k\|^2). 
\end{align}
\end{lemma}

Applying lemma \eqref{thm:key-bound-ho} to a Newton-Anderson step yields 
\begin{align}\label{na-null-bd-ho}
P_Ne_{k+1}=\bigg(\frac{d}{d+1}\bigg) P_Ne_k^{\alpha}+(\hat{T}_kP_Re_k)^{\alpha}+q_{k-1}^k.
\end{align}

\begin{lemma}
Under the assumptions of lemma \eqref{thm:key-bound-ho}, strong N-pairs are automatically compatible. Further, writing $x_{k+1}=(1-\gamma_{k+1})(x_k+w_{k+1})+\gamma_{k+1}(x_{k-1}+w_k)$ and $e_k=x_k-x^*$, the following bound holds
\begin{align}\label{null-error-ho}
\|P_Ne_{k+1}\|\leq \theta_{k+1}\bigg(\frac{d}{d+1}\bigg)\bigg(\frac{1+r_{k+1}^e}{1-r_{k+1}^e}\bigg)\|P_Ne_k\|.
\end{align}
\end{lemma}

\begin{proof}
Assume $r_{k+1}^e<<1$. Then if $(x_k,x_{k-1})$ is a strong N-pair, we have 
\begin{align}
\|P_Ne_{k+1}\|\leq (1+r_{k+1}^e)\bigg(\frac{d}{d+1}\bigg)\|P_Ne_k^{\alpha}\|.
\end{align}
By lemma \eqref{thm:key-bound-ho}, it follows that 
\begin{align}
\|P_Ne_{k+1}\|\leq d\bigg(\frac{1+r_{k+1}^e}{1-r_{k+1}^e}\bigg)\|w_{k+1}^{\alpha}\|= d\theta_{k+1}\bigg(\frac{1+r_{k+1}^e}{1-r_{k+1}^e}\bigg)\|w_{k+1}\|.
\end{align}
Applying \eqref{key-w-bd-ho} completes the proof.
\end{proof}

While $d$ may grow very large for very high-order roots, we still feel it's reasonable to say that strong N-pairs are automatically compatible because the final bound on $P_Ne_{k+1}$ seen in \eqref{null-error-ho} contains $d(d+1)^{-1}$ rather than $d$ alone. This ensures that $\|P_Ne_{k+1}\|=\bigo(\theta_{k+1}\|P_Ne_k\|)$ where the constant is of moderate size. 
Thus $\theta_{k+1}$ can lead to significant acceleration just as in the order one case. It's worth noting that this acceleration can occur without any a priori knowledge of the order of the root. Theorem \ref{thm:convergence-ho} is the generalization of \ref{thm:convergence-thm} for higher order roots. The proofs are nearly identical due to the similar structure of \ref{null-bound-ho} and \ref{keybound1}, and therefore we simply state \ref{thm:convergence-ho}. Analogously to the first order root case, we let $\hat{D}^{d}_N(x)=(d!)^{-1}P_Nf^{(d+1)}(x^*)P_N(x-x^*)^{d}P_N$.

\begin{theorem}\label{thm:convergence-ho}
Let $\dim N=1$, and let $\hat{D}^{d+1}_N(x)$ be invertible as a map on $N$ for all $P_N(x-x^*)\neq 0$. Let $W_k=W(\|e_k\|,\sigma_k,x^*)$. If $x_0$ is chosen so that $\sigma_0<\hat{\sigma}$ and $\|e_0\|<\hat{\rho}$, for sufficiently small $\hat{\sigma}$ and $\hat{\rho}$, $x_1=x_0+w_1$, and $x_{k+1}=L_{k+1}^{\lambda}(x_k+w_{k+1},x_{k-1}+w_k)$ for $k\geq 1$, then $W_{k+1}\subset W_0$ for all $k\geq 0$ and $x_k\to x^*$. That is, $\set{x_k}$ remains well-defined and converges to $x^*$. Furthermore, there exist constants $c>0$ and $\kappa\in(d/(d+1),1)$ such that 
\begin{align}
\|P_Re_{k+1}\|&\leq c\max\{|1-\lambda_{k+1}\gamma_{k+1}|\,\|e_k\|^{d+1},|\lambda_{k+1}\gamma_{k+1}|\,\|e_{k-1}\|^{d+1}\}\\ \|P_Ne_{k+1}\|&\leq \kappa\theta_{k+1}^{\lambda}\|P_Ne_k\|.  \end{align}
\end{theorem}

\section{Numerical Results}\label{numerical-results}
In this section we compare the performance of a few variations of Newton-Anderson 
(N.Anderson), including standard N.Anderson, and the projected Levenberg-Marquardt method from \cite{KaYaFu04}. The variations on N.Anderson are $\gamma$-N.Anderson, that is, N.Anderson with $\gamma$-safeguarding from algorithm \ref{alg:gamma-safeguarding-pseudocode}, N.Anderson with an Armijo linesearch (see, for instance, \cite{Armijo66}), and $\gamma$-Armijo-N.Anderson, that is, we use algorithm \ref{alg:gamma-safeguarding-pseudocode} in tandem with a linesearch. The linesearch is only carried out if a given step fails to reduce the residual by a factor of 0.99. It should be noted that in practice, $\gamma$-safeguarding does not appear necessary for convergence, as is demonstrated below. Though it can recover and even improve convergence for some problems. 
There is only one problem for which N.Anderson did not converge. Both Armijo-N.Anderson and $\gamma$-N.Anderson converged well for that problem.

Our first two test problem feature nonlinear systems of order $n = 10^4$: 
the Chandrasekhar H-equation from \cite{chandra60}, and a modified version of the multivariate polynomial seen in example three on page 45 of \cite{LiZhi22}. The Chandrasekhar H-equation is a familiar benchmark problem from the literature concerning Newton's method, in particular Newton's method in the presence of singularities. The multivariate polynomial is a scaleable singular problem where we can easily adjust the order of the root in order to demonstrate the theory for higher order roots. Here we mean the order of the root as defined in \ref{sec:high_order_roots}. We then apply the methods to various smaller-scale benchmark problems, both singular and nonsingular, from the literature. All test problems are square systems. 

All computations were performed in Octave on an M1 Macbook Pro. The iterations were terminated when $||f(x_k)||<10^{-8}$ or the number of iterations exceeded fifty, in which case we say the algorithm failed to converge. The parameters for the projected Levenberg-Marquardt method are mostly those of \cite{KaYaFu04}. Our implementation differs from theirs in that we terminate when $\|f(x_k)\|<10^{-8}$, and our Armijo linesearch parameters may differ. In our search, the step size is $1/2$ and the damping parameter is $10^{-4}$. Thus we seek the smallest $j$ such that $g(x+(1/2)^jd)\leq g(x)+10^{-4}(1/2)^j g'(x)d$, where $g(x)=\|f(x)\|^2$ and $d$ is the search direction. For the N.Anderson method and variants, $d_k = w_{k+1}-\gamma_{k+1}(x_k-x_{k-1}+w_{k+1}-w_k)$, and we scale the step size 1/2 by 3/10 at each iteration for each problem except for Dayton10. That is, we seek the smallest $j$ such that $g(x+(1/2)(3/10)^jd)\leq g(x)+10^{-4}(1/2)(3/10)^jg'(x)d$. For Dayton10, we took the step size to be 4/5 and scaled by 3/10 at each iteration. When we want to emphasize the $\gamma$-safeguarding parameter $r$, we write $\gamma$-N.Anderson($r$).  

\subsection{Singular Problems}
\subsubsection{The Chandrasekhar H-Equation}
The Chandrasekhar H-equation is an important benchmark problem for Newton and Newton-like methods (see \cite{DeKe82,Ke18,KeSu83,posc20}) defined by the integral equation
\begin{align}\label{fun:chh}
    F(H)(\mu):=H(\mu)-\bigg(1-\frac{\omega}{2}\int_0^1 \frac{\mu H(\nu)\,d\nu}{\mu+\nu}\bigg)^{-1}=0.
\end{align}
As discussed in \cite{Ke18}, \cref{fun:chh} 
admits real solutions for parameter $\omega \in [0,1]$. The derivative
at each solution is invertible for $\omega \in [0,1)$, and features a one dimensional
nullspace at the bifurcation point, $\omega = 1$.
As shown in \cite{DeKe82}, \eqref{fun:chh} with $\omega = 1$
satisfies the assumptions of \cref{thm:convergence-thm}. 
Following \cite{Ke18}, we discretize  \eqref{fun:chh}
by the composite midpoint rule with $n = 10^4$ nodes, yielding a 
discrete system in $\mathbb{R}^n$. We take $x^0$ to be the vector of ones. Below are the results. 

\begin{figure}[H]
	\centering
	\includegraphics[width=61mm]{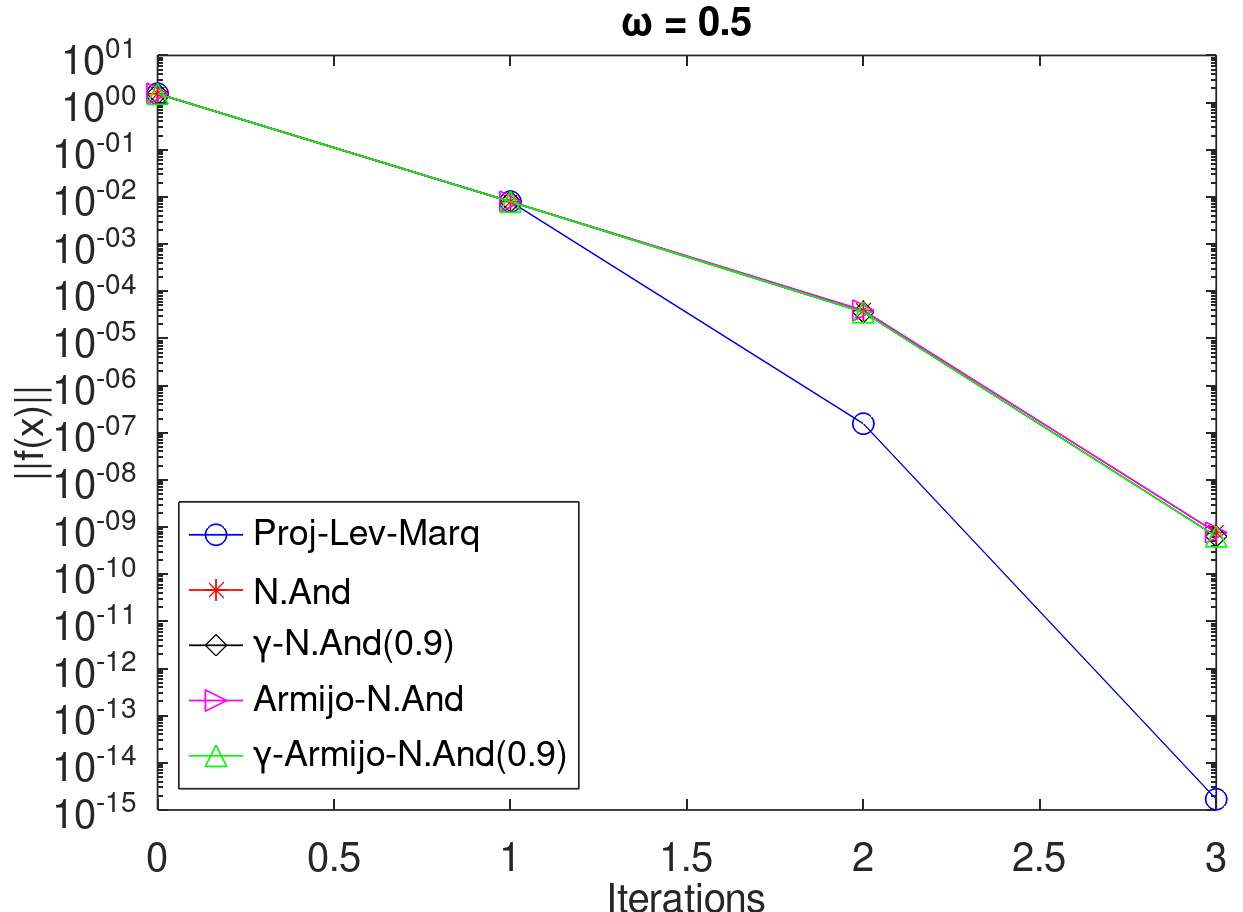}
	\includegraphics[width=61mm]{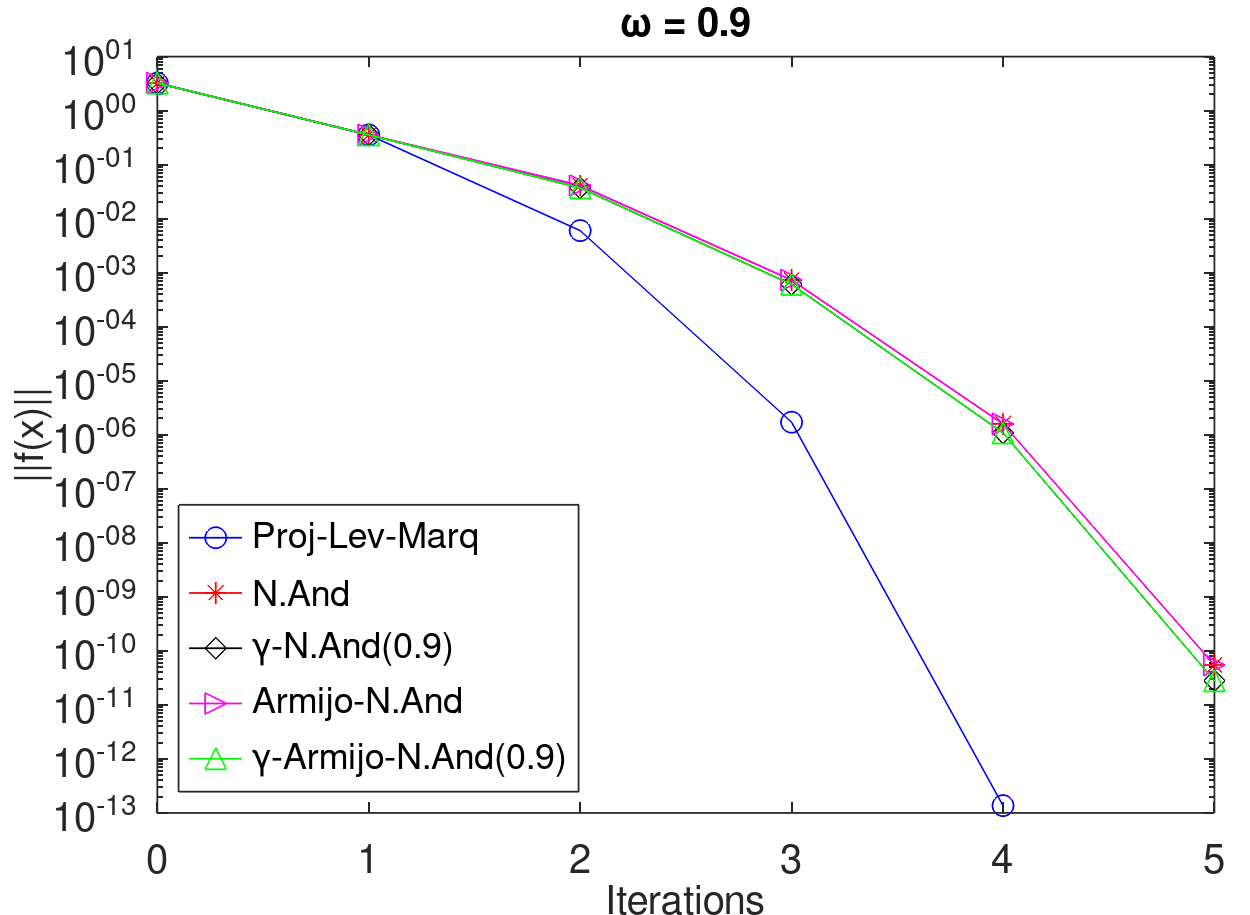}
	\caption{Left: Residual history for the Chandrasekhar H-equation with parameter $\omega=0.5$. Right: Residual history for Chandrasekhar H-equation with $\omega=0.9$.}
\end{figure}

\begin{figure}[H]
	\centering
	\includegraphics[width=61mm]{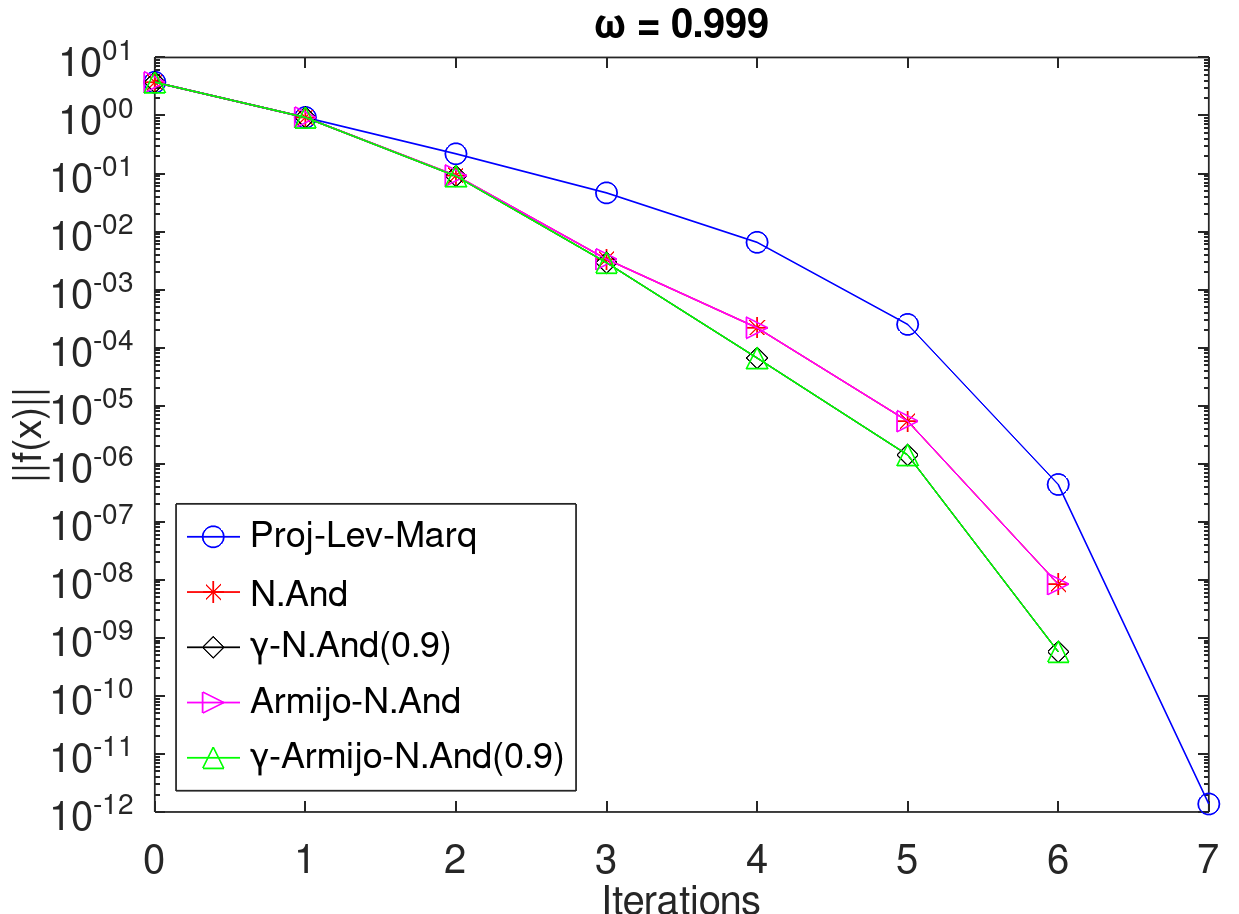}
	\includegraphics[width=61mm]{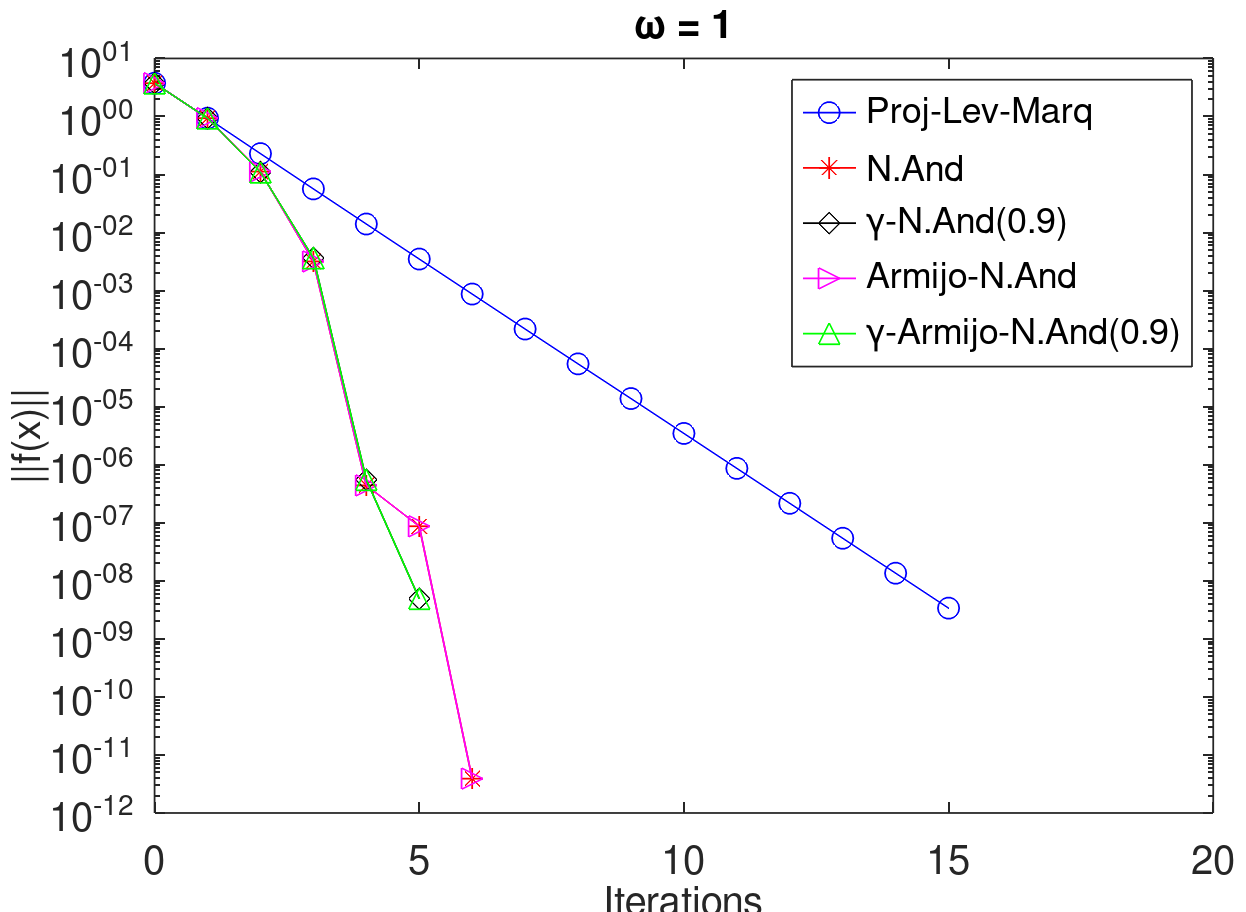}
	\caption{Left: Residual history for the Chandrasekhar H-equation with parameter $\omega=0.999$. Right: Residual history for Chandrasekhar H-equation with $\omega=1$.}
\end{figure}

No line search steps were taken for the N.Anderson variations. Consequently, the residual history for N.Anderson and Armijo-N.Anderson are identical, as are $\gamma$-N.Anderson and $\gamma$-Armijo-N.Anderson. We obesrve that 
Projected-Levenberg-Marquardt
converges moderately faster than the N.Anderson variations up to 
$\omega = 0.999$, however the iteration count to converge to the given tolerance varies by only one. For $\omega = 1$ the problem is singular, and the N.Anderson variations
display substantially better performance.
For each case, the variatons of N.Anderson all perform similarly
to each other with $\gamma$-N.Anderson slightly outperforming N.Anderson for $\omega=0.999$. For comparison, we note that from the same initial iterate, the standard Newton method converged in 4 iterations for $\omega=0.5$, 5 for $\omega=0.9$, 8 for $\omega=0.999$, and 17 for $\omega=1$.

\subsubsection{A Multivariate Polynomial}
Here we apply the methods mentioned at the beginning of this section to 
\begin{align}\label{fun:multipoly}
\scalebox{0.93}{
	$f(x_1,...,x_n) = (x_1^2+x_1-x_2^k,x_2^2+x_2-x_3^k,...,x_{n-1}^2+x_{n-1}-x_n^k,x_n^k), \hspace{1em} n=10^4.$}
\end{align}
The zero vector is a root of order $k-1$. That is, $P_N f^{(j)}(x^*)=0$ for $j=0,1,...,k-1$, and is nonzero for $j=k$. Here $x^*$ is the zero vector. We choose $x^0$ such that $x^0_n = 0.9$, and $x^0_j=0.3$ for all $j=1,2,...,n-1$. 
\begin{figure}[H]
	\centering
	\includegraphics[width=41mm]{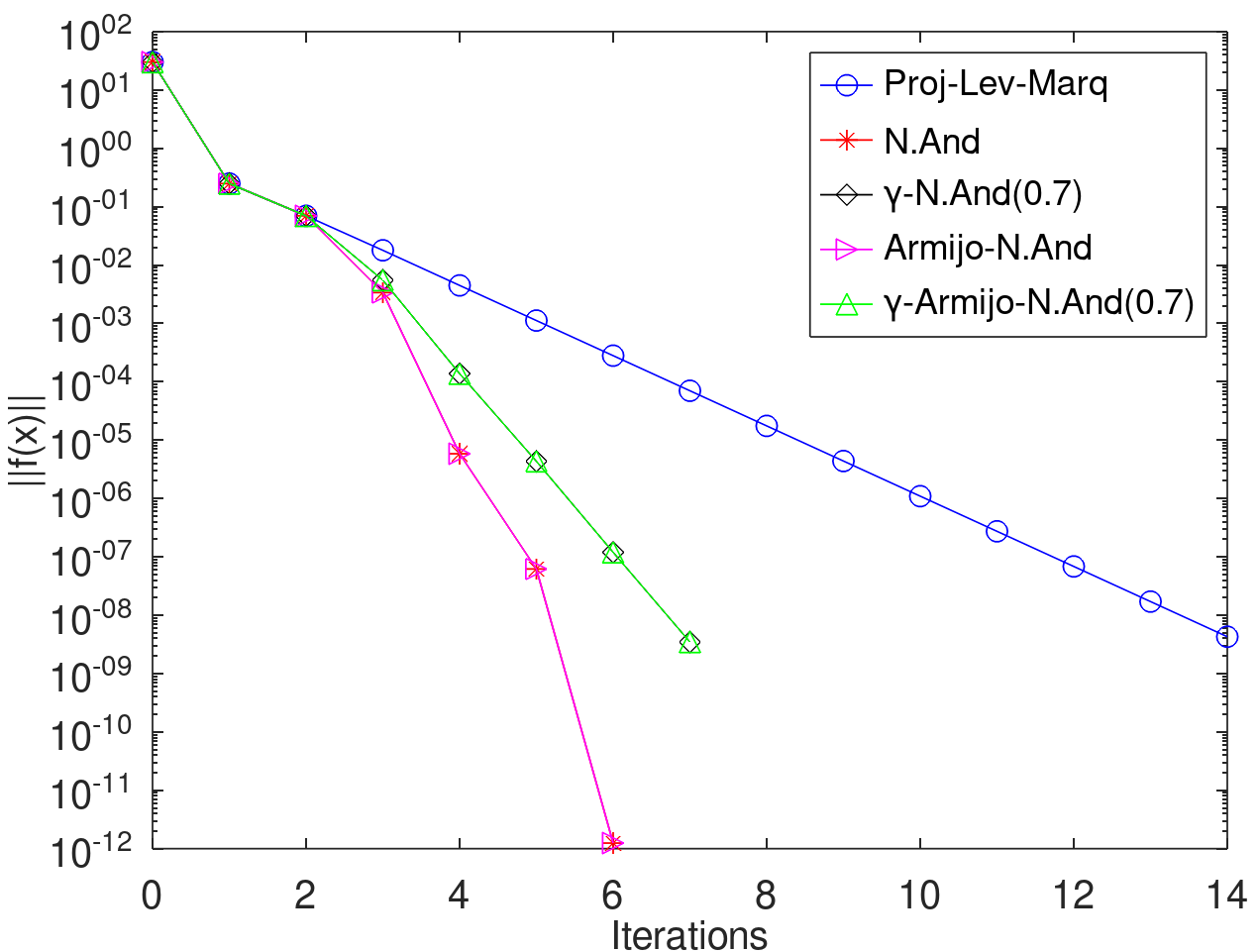}
	\includegraphics[width=41mm]{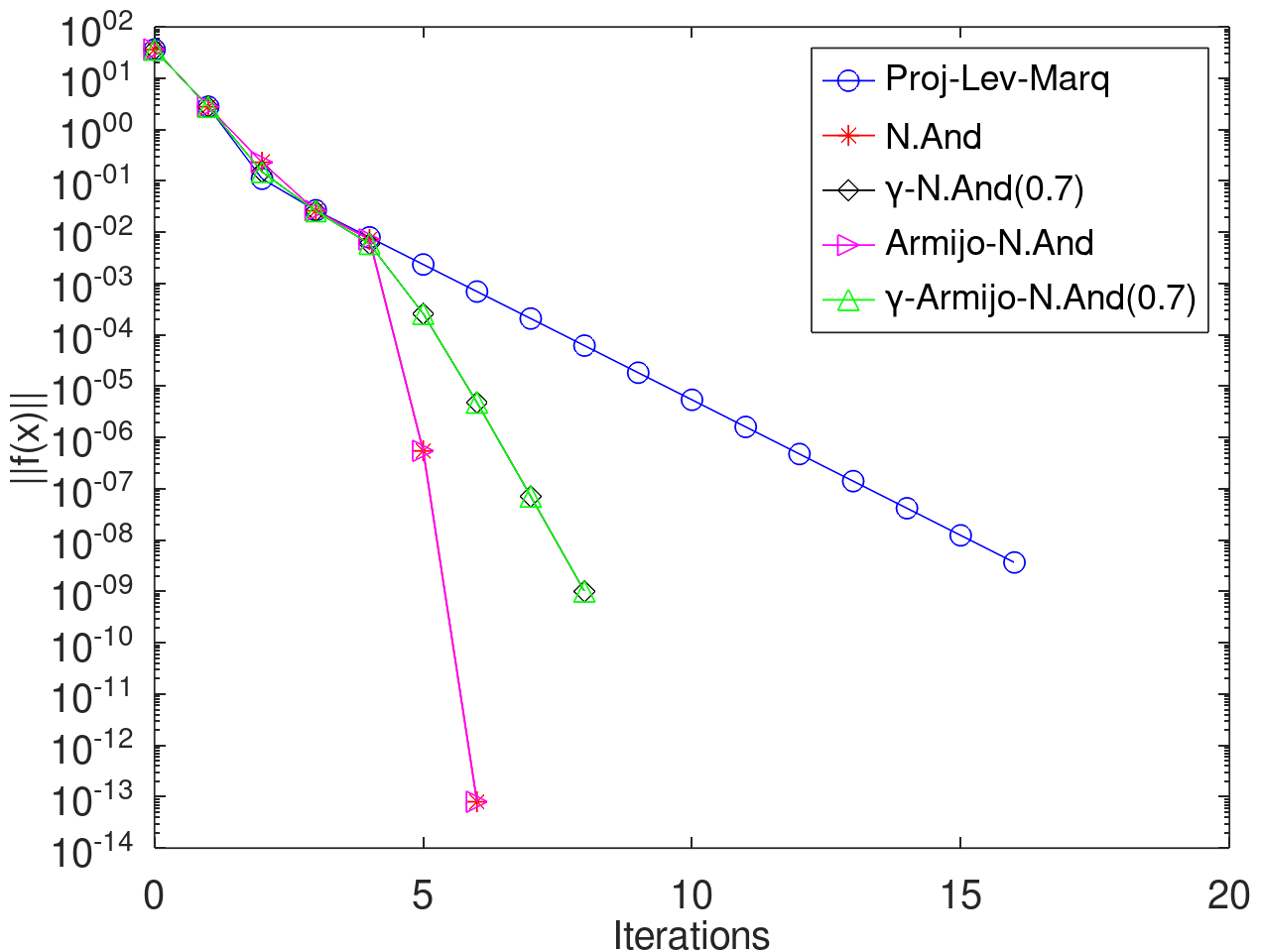}
	\includegraphics[width=41mm]{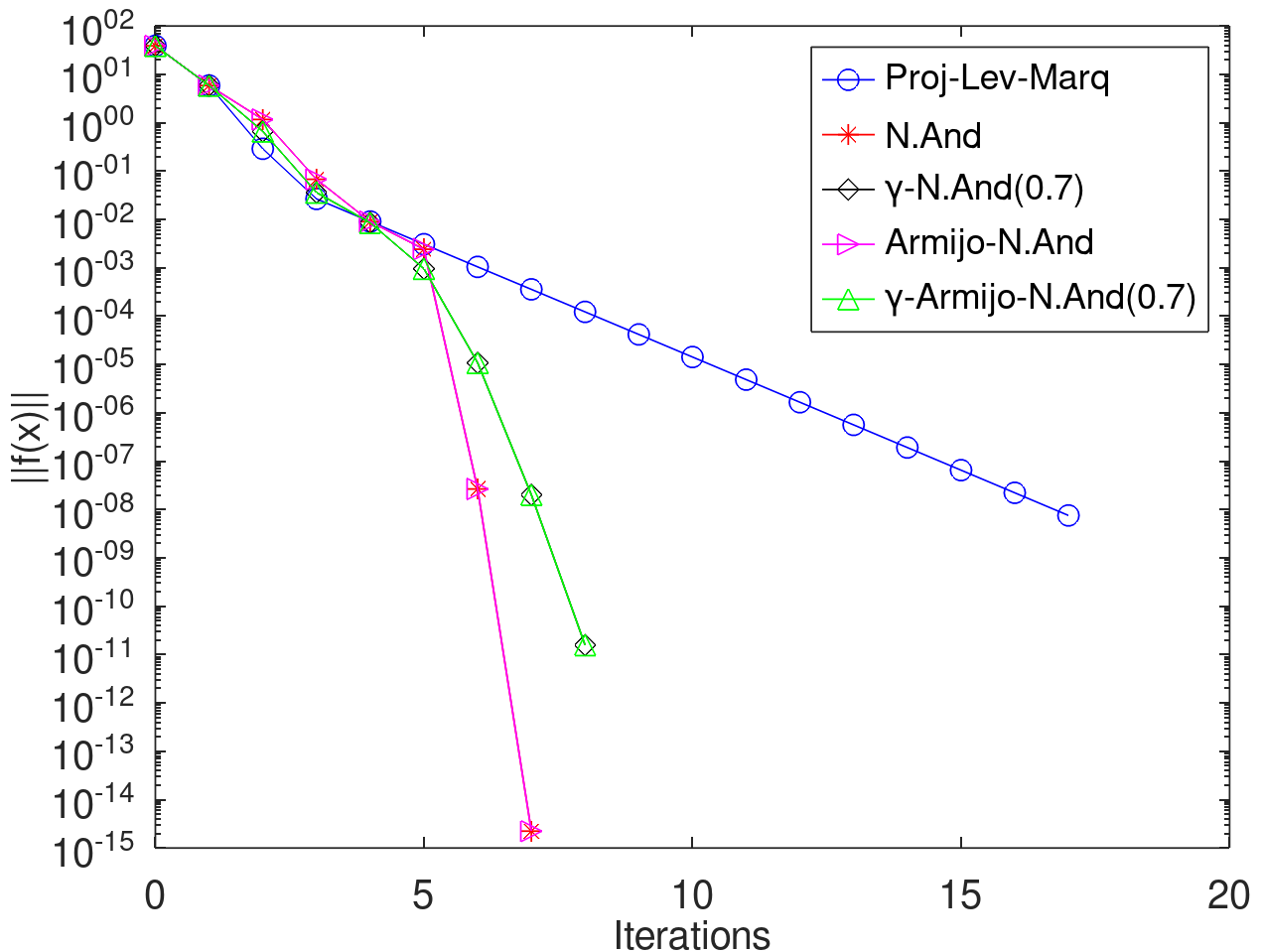}
	\caption{\label{fig:himultipoly}Left: Residual history for $k=2$, in which case the zero vector is a root of order one. Middle: Residual history for $k=3$. Right: Residual history for $k=7$.}
\end{figure}

The results in figure \ref{fig:himultipoly} demonstrate the efficacy of N.Anderson and its variations in the presence of higher order roots. As indicated in the legends, the result of running $\gamma$-N.Anderson with $\gamma$-safeguarding parameter $r=0.7$ are shown. We note that upon setting $r=0.9$, the convergence was very similar to that of standard N.Anderson. No line search steps were taken for the N.Anderson variations. Consequently, the residual history for N.Anderson and Armijo-N.Anderson are identical, as are $\gamma$-N.Anderson and $\gamma$-Armijo-N.Anderson. The Projected-Levenberg-Marquardt method is seen to converge linearly, and converge slightly slower as the order of the root increases. The convergence of the variations of N.Anderson, however, is less affected by the order of the root. With $\gamma$-safegurading parameter $r=0.7$, $\gamma$-N.Anderson converges slower than standard N.Anderson, but it still exhibits fast convergence compared to the Projected-Levenberg-Marquardt method as the order of the root increases. From the same initial iterate, the standard Newton method converged in 15 iterations when $k=2$, 17 when $k=3$, and 18 when $k=7$.

\subsection{Small-Scale Benchmark Problems}
Numerical results for the small-scale problems can be found in appendix \ref{appendix-a}.
The test problems in table \ref{table:nonsingular} correspond to test problems $14.1.1-14.2.6$ in \cite{floudas_handbook}. Following \cite{KaYaFu04}, we take $x_0$ to be the vector of lower bounds found in \cite{floudas_handbook} for these problems. 
The test problems in table \ref{table:singular} are taken from various sources in the literature, which we include by the name of the problem in the first column. The largest of these problems has dimension $n=8$. The initial iterate is taken from the source for each problem except Dayton10, for which we take $x_0 = (0.9,0.9,0.9,0.9)^T.$ The fifth column $||f(x)||$ is the value of the residual upon termination. The final column, following \cite{KaYaFu04}, records the number of standard projected Levenberg-Marquardt steps taken, the number of Armijo linesearch steps, and the number of projected gradient step (see Algorithm 3.12 and the first paragraph of section 4 in \cite{KaYaFu04}). We also use this column to count the number of Armijo linesearch steps taken in Armijo-N.Anderson and $\gamma$-Armijo-N.Anderson. The number of function evaluations, denoted by $f$-evals, is not recorded for N.Anderson or $\gamma$-N.Anderson since for these algorithms the number of function evaluations is always the number of iterations plus one.

For the nonsingular problems,
as shown in table \ref{table:nonsingular},
 $\gamma$-N.Anderson performs comparably to 
the Levenberg-Marquardt method in all but two problems: Eq-Combustion and Robot Kin. 
Sys.  
All methods perform similarly with the Ferraris-Tronconi function. For the Bullard-Biegler function, for our chosen initial iterate, N.Anderson fails to converge.
We recover convergence by applying $\gamma$-safeguarding, a linesearch, or both, and $\gamma$-N.Anderson outperforms Armijo-N.Anderson. In fact, $\gamma$-N.Anderson is seen to outperform or match Armijo-N.Anderson for each of the nonsingular test problems. 

For the singular problems, 
as shown in table \ref{table:singular},
 N.Anderson and its variations all outperform Projected-Levenberg-Marquardt, except for Dayton10 where $\gamma$-Armijo-N.Anderson failed to converge. For most of the problems, applying $\gamma$-safeguarding, a linesearch, or both yields similar results. 
Although the linesearch can yield better results than $\gamma$-safeguarding, the latter requires fewer function evaluations than a linesearch; and in some cases, far fewer.

\section{Conclusion}\label{conclusion}
We have presented an analysis of and convergence theory for the N.Anderson method in a general Euclidean space when the derivative has a nontrivial null space at the root. 
Efficient and robust methods for this problem class are of continuing 
importance as they naturally arise at bifurcation points in parameter-dependent 
mathematical models.
As motivated by previous numercial benchmarking results where greater algorithmic 
depths were explored but generally not found beneficial, we restricted our attention 
to an algorithmic depth of one
for Anderson accelerated Newton iterations. Greater algorithmic depths and methods
for problems with nullspaces of higher dimension may be explored
in future work if they are found to be advantageous in relevant applications. 

We showed that near the null space, the component of the error along the null space is accelerated by a factor determined by the success of the minimization step in the N.Anderson algorithm, which is measured by the optimization gain. In other words, 
it is shown that the region in which standard Newton exhibits linear convergence is precisely the region most susceptible to acceleration, 
demonstrated by a substantial decrease in the convergence 
rate when the optimization gain is small. 
We developed a novel and theoretically supported safeguarding scheme that when applied to N.Anderson ensures local convergence under the same conditions as standard Newton; moreover, the rate of convergence 
is improved in general, depending on the optimization gain. 
The theory was demonstrated with benchmark problems from the literature.

\section{Acknowledgements} 
MD and SP are supported in part by the National Science Foundation NSF-DMS 2011519.

\bibliographystyle{plain}
\bibliography{newt_and_sing_pts}

\appendix

\section{Numerical Results for Small-Scale Problems}
\label{appendix-a}
\begin{table}[H]
\begin{tabular}{|c|c|c|c|c|c|}
\hline
Problem &Algorithm &Iterations &f-evals &$||f(x)||$ &LM/LS/PG \\
\hline
\multirow{5}{2cm}{Himmelbau} &Proj-Lev-Marq &6 &7 &2.842e-14 &6/0/0 \\
&N.Anderson &8 &- &7.105e-15 &- \\
&$\gamma$-N.Anderson(0.5) &6 &- &5.309e-12 &- \\
&Armijo-N.Anderson &8 &9 &7.105e-15 &-/0/- \\
&$\gamma$-Armijo-N.Anderson(0.5) &7 &8 &5.309e-12 &-/0/- \\
\hline
\multirow{5}{2cm}{Eq-Combustion} &Proj-Lev-Marq &11 &28 &8.200e-14 &8/3/0 \\
&N.Anderson &35 &- &8.510e-12 &- \\
&$\gamma$-N.Anderson(0.5) &17 &- &3.092e-09 &- \\
&Armijo-N.Anderson &18 &56 &2.136e-10 &-/2/- \\
&$\gamma$-Armijo-N.Anderson(0.5) &17 &18 &3.092e-09 &-/0/- \\
\hline
\multirow{5}{2cm}{Bullard-Biegler} &Proj-Lev-Marq &13 &26 &2.270e-10 &10/3/0 \\
&N.Anderson &F &- &- &- \\
&$\gamma$-N.Anderson(0.5) &11 &- &1.799e-12 &- \\
&Armijo-N.Anderson &20 &202 &1.212e-10 &-/12/- \\
&$\gamma$-Armijo-N.Anderson(0.5) &13 &34 &1.629e-11 &-/6/- \\
\hline
\multirow{5}{2cm}{Ferraris-Tronconi} &Proj-Lev-Marq &4 &5 &5.339e-14 &4/0/0 \\
&N.Anderson &4 &- &6.937e-11 &- \\
&$\gamma$-N.Anderson(0.5) &4 &- &6.008e-11 &- \\
&Armijo-N.Anderson &4 &5 &6.937e-11 &-/0/- \\
&$\gamma$-Armijo-N.Anderson(0.5) &4 &5 &6.008e-11 &-/0/- \\
\hline
\multirow{5}{2cm}{Brown's Al. Lin.} &Proj-Lev-Marq &9 &10 &1.638e-14 &9/0/0 \\
&N.Anderson &19 &- &5.137e-10 &- \\
&$\gamma$-N.Anderson(0.5) &11 &- &1.441e-11 &- \\
&Armijo-N.Anderson &11 &19 &1.286e-10 &-/2/- \\
&$\gamma$-Armijo-N.Anderson(0.5) &11 &12 &1.441e-11 &-/0/- \\
\hline
\multirow{5}{2cm}{Robot Kin. Sys.} &Proj-Lev-Marq &5 &6 &6.404e-10 &5/0/0 \\
&N.Anderson &9 &- &7.390e-14 &- \\
&$\gamma$-N.Anderson(0.5) &8 &- &4.290e-14 &- \\
&Armijo-N.Anderson &9 &10 &7.390e-14 &-/0/- \\
&$\gamma$-Armijo-N.Anderson(0.5) &8 &9 &4.290e-14 &-/0/- \\
\hline
\end{tabular}
\caption{\label{table:nonsingular}Nonsingular Benchmark Problems}
\end{table}

\begin{table}[H]

\begin{tabular}{|c|c|c|c|c|c|}
\hline
Problem &Algorithm &Iterations &f-evals &$||f(x)||$ &LM/LS/PG \\
\hline
\multirow{5}{2cm}{Decker1\cite{DeKeKe83}} &Proj-Lev-Marq &15 &16 &3.559e-09 &15/0/0 \\
&N.Anderson &9 &- &2.698e-12 &- \\
&$\gamma$-N.Anderson(0.9) &8 &- &4.187e-09 &- \\
&Armijo-N.Anderson &9 &10 &2.698e-12 &-/0/- \\
&$\gamma$-Armijo-N.Anderson(0.9) &8 &9 &4.187e-09 &-/0/- \\
\hline
\multirow{5}{2cm}{Decker2\cite{DeKe80-1}} &Proj-Lev-Marq &16 &17 &3.356e-09 &16/0/0 \\
&N.Anderson &7 &- &3.118e-09 &- \\
&$\gamma$-N.Anderson(0.9) &7 &- &8.659e-09 &- \\
&Armijo-N.Anderson &7 &8 &3.118e-09 &-/0/- \\
&$\gamma$-Armijo-N.Anderson(0.9) &7 &8 &8.659e-09 &-/0/- \\
\hline
\multirow{5}{2cm}{Ojika1\cite{Oj88}} &Proj-Lev-Marq &28 &59 &2.419-09 &18/10/0 \\
&N.Anderson &19 &- &1.620e-09 &- \\
&$\gamma$-N.Anderson(0.9) &17 &- &5.990e-09 &- \\
&Armijo-N.Anderson &17 &69 &4.214e-09 &-/5/- \\
&$\gamma$-Armijo-N.Anderson(0.9) &17 &64 &6.162e-09 &-/4/- \\
\hline
\multirow{5}{2cm}{Ojika2\cite{OjWaMi83} } &Proj-Lev-Marq &13 &14 &2.909e-09 &13/0/0 \\
&N.Anderson &7 &- &3.096e-09 &- \\
&$\gamma$-N.Anderson(0.9) &7 &- &7.165e-09 &- \\
&Armijo-N.Anderson &7 &8 &3.096e-09 &-/0/- \\
&$\gamma$-Armijo-N.Anderson(0.9) &7 &8 &7.165e-09 &-/0/- \\
\hline
\multirow{5}{2cm}{Pollock1\cite{posc20}} &Proj-Lev-Marq &14 &15 &3.991e-09 &14/0/0 \\
&N.Anderson &5 &- &1.656e-10 &- \\
&$\gamma$-N.Anderson(0.9) &5 &- &8.268e-10 &- \\
&Armijo-N.Anderson &5 &6 &1.656e-10 &-/0/- \\
&$\gamma$-Armijo-N.Anderson(0.9) &5 &6 &8.268e-10 &-/0/- \\
\hline
\multirow{5}{2cm}{Dayton10\cite{DaZe05}} &Proj-Lev-Marq &F &- &- &- \\
&N.Anderson &11 &- &3.294e-10 &- \\
&$\gamma$-N.Anderson(0.5) &14 &- &5.157e-09 &- \\
&Armijo-N.Anderson &15 &140 &2.153e-11 &-/4/- \\
&$\gamma$-Armijo-N.Anderson(0.5) &F &- &- &-/-/- \\
\hline
\multirow{5}{2cm}{Hueso1\cite{HMT09}} &Proj-Lev-Marq &13 &14 &3.701e-09 &13/0/0 \\
&N.Anderson &12 &- &5.417e-09 &- \\
&$\gamma$-N.Anderson(0.9) &12 &- &4.918e-09 &- \\
&Armijo-N.Anderson &9 &39 &7.184e-09 &-/1/- \\
&$\gamma$-Armijo-N.Anderson(0.9) &10 &40 &6.264e-09 &-/1/- \\
\hline
\multirow{5}{2cm}{Hueso6\cite{HMT09}} &Proj-Lev-Marq &16 &17 &3.531e-09 &16/0/0 \\
&N.Anderson &6 &- &3.459e-10 &- \\
&$\gamma$-N.Anderson(0.9) &6 &- &6.590e-09 &- \\
&Armijo-N.Anderson &6 &7 &3.459e-10 &-/0/- \\
&$\gamma$-Armijo-N.Anderson(0.9) &6 &7 &6.590e-09 &-/0/- \\
\hline

\end{tabular}
\caption{\label{table:singular}Singular Benchmark Problems}
\end{table}

\section{Proof of Convergence Theorem}
\label{appendix-b}

We'll now prove theorem \ref{thm:convergence-thm}. We restate it here for convenience. 

\begin{theorem}
Let $\dim N=1$, and let $\hat{D}_N(x)$ be invertible as a map on $N$ for all $P_N(x-x^*)\neq 0$. Let $W_k=W(\|e_k\|,\sigma_k,x^*)$. If $x_0$ is chosen so that $\sigma_0<\hat{\sigma}$ and $\|e_0\|<\hat{\rho}$, for sufficiently small $\hat{\sigma}$ and $\hat{\rho}$, $x_1=x_0+w_1$, and $x_{k+1}=L_{k+1}^{\lambda}(x_k+w_{k+1},x_{k-1}+w_k)$ for $k\geq 1$, then $W_{k+1}\subset W_0$ for all $k\geq 0$ and $x_k\to x^*$. That is, $\set{x_k}$ remains well-defined and converges to $x^*$. Furthermore, there exist constants $c_4>0$ and $\kappa\in(1/2,1)$ such that 
\begin{align}
  \|P_Re_{k+1}\|&\leq
	 c_4\max\set{|1-\lambda_{k+1}\gamma_{k+1}|\,\|e_k\|^2,|\lambda_{k+1}\gamma_{k+1}|\,\|e_{k-1}\|^2}\label{range-bd-main-app}\\
   \|P_Ne_{k+1}\|&< \kappa\theta_{k+1}^{\lambda}\|P_Ne_{k}\|\label{null-rate-of-convergence-main-app}
\end{align}
for all $k\geq 1$.
\end{theorem}

\begin{proof}
Throughout this proof, we assume $\gamma$-safeguarding is in use, and will simply write ``$\gamma_{k+1}$" in place of ``$\lambda_{k+1}\gamma_{k+1}$". Further, it's assumed that $\gamma_k<1$ for all $k$. If $\gamma_k=1$, we take a standard Newton step. We'll let $c$ denote intermediate constants that will eventually be absorbed into another. The bound in \eqref{range-bd-main-app} is just a restatement of \eqref{range-comp-bound} with $\gamma$-safeguarding. It will follow for all $k$ once we prove that $\set{x_k}$ is well-defined. The 
remainder of the proof proceeds
as follows. We'll first show inductively that $W_k\subset W_0$ for all $k$, thus ensuring $\set{x_k}$ is well-defined. Along the way we'll establish \eqref{null-rate-of-convergence-main-app}, and define a sequence $\set{\eta_k}$ such that $\|e_{k+1}\|<\eta_k\|e_0\|$. Then we'll show $\eta_k\to 0$, proving convergence. We begin with $x_2$. 

It's known from standard Newton theory (cf \cite{DeKeKe83}) that if $P_Ne_0\neq 0$ and $\sigma_0$ and $\|e_0\|$ are sufficiently small there exists constants $c>0$, $\psi_0\in (0,1)$, $\eta_0\in (0,1)$, and $s_0\in (0,1)$ such that $\|P_Re_1\|<c\|e_0\|^2$, $\|P_Ne_1\|<\psi_0 \|P_Ne_0\|$, $\|e_1\|<\eta_0\|e_0\|$, and $\sigma_1<s_0\sigma_0$. Moreover, $\hat{D}_N(x_0)$ and $\hat{D}_N(x_1)$ are invertible as a maps on $N$. Hence \eqref{newton-anderson-error} and \eqref{theta-w-expansion} hold. Apply $P_N$ to \eqref{newton-anderson-error} with $k=1$ and pull out $\|(1/2)P_Ne_1^{\alpha}\|$ to get 
\begin{align}
    \|P_Ne_{2}\|&\leq \bigg(1+\frac{\|T_1P_Re_1^{\alpha}+q_0^1\|}{\|P_Ne_1^{\alpha}/2\|}\bigg)\|P_Ne_1^{\alpha}/2\|\\
    &\leq \bigg(1+\frac{\|T_1P_Re_1^{\alpha}+q_0^1\|}{(1-\overset{\sim}{\nu}_{2})\max\set{|1-\gamma_{2}|\,\|P_Ne_1\|,|\gamma_{2}|\,\|P_Ne_0\|}/2}\bigg)\|P_Ne_1^{\alpha}/2\|.\nonumber
\end{align}
Lemma \ref{gamma-safeguarding} guarantees that with $\gamma$-safeguarding, $1/(1-\overset{\sim}{\nu}_{k+1})$ remains bounded independent of $k$. Moreover,  
\begin{align}\label{bound-on-hot}
\scalebox{0.94}{
    $\|T_1P_Re_1^{\alpha}+q_0^1\|\leq c_5\max\set{|1-\gamma_2|\,\|P_Re_1\|,|\gamma_2|\,\|P_Re_0\|,|1-\gamma_2|\,\|e_1\|^2,|\gamma_2|\,\|e_0\|^2}.$}
\end{align} 
The constant $c_5$ here depends on
$f$. Since for any real numbers $a$ and $b$, $1/\max\set{a,b}\leq \min\set{1/a,1/b}$, it follows that 
\begin{align}
    \|T_1P_Re_1^{\alpha}+q_0^1\|/\|P_Ne_1^{\alpha}\|\leq c\max\{\|P_Re_i\|/\|P_Ne_i\|,\|e_i\|^2/\|P_Ne_i\|\},
\end{align}
with $i=0$ or $i=1$ depending on the value of the right hand side of \eqref{bound-on-hot}. The constant $c$ here depends on $f$ and the value of $r$ in \cref{alg:gamma-safeguarding-pseudocode}. Noting that $\|e_i\|^2/\|P_Ne_i\|<(1+\sigma_i)\|e_i\|$, it follows that
\begin{align}
    \|P_Ne_2\|\leq(\,1+C_1(\sigma_{0,1},\|e_{0,1}\|)\,)\|(1/2)\|P_Ne_1^{\alpha}\|,
\end{align} where $C_1(\sigma_{0,1},\|e_{0,1}\|)=c\max\set{\sigma_0,\sigma_1,\|e_0\|,\|e_1\|}$. 
Through an analogous argument applied to \eqref{theta-w-expansion} and the definition of $\theta_{k+1}$, we can bound $|(1/2)P_Ne_1^{\alpha}\|$ in terms of $\|w_2^{\alpha}\|$ to obtain $\|P_Ne_2\|\leq \big((\,1+C_1(\sigma_{0,1},\|e_{0,1}\|)\,)(1-C_2(\sigma_{0,1},\|e_{0,1}\|)\,)^{-1}\big)\theta_{2}\|w_{2}\|.$ Using \eqref{keybound2} to bound $\|w_2\|$ in terms of $\|P_Ne_1\|$ yields
    \begin{align}
    \|P_Ne_2\|\leq \bigg(\frac{(1/2)(1+C_1(\sigma_{0,1},\|e_{0,1}\|)\,)(1+C_3(\sigma_{0,1},\|e_{0,1}\|)}{1-C_2(\sigma_{0,1},\|e_{0,1}\|)}\bigg)\theta_{2}\|P_Ne_1\|.
    \end{align}
We'll write $(1+C_1(\sigma_{0,1},\|e_{0,1}\|)(1+C_3(\sigma_{0,1},\|e_{0,1}\|)=(1+C_1(\sigma_{0,1},\|e_{0,1}\|)$.  
 Reducing $\|e_0\|$ and $\sigma_0$ further if necessary, we have 
 \begin{align}\label{main-null-bound}
     \|P_Ne_2\|<\kappa\theta_2\|P_Ne_1\|
 \end{align}
where $1/2<\kappa<1$.
 Note that $\|P_Ne_1\|\leq \psi_0\|P_Ne_0\|\leq \psi_0(1-\sigma_0)^{-1}\|e_0\|$. Writing $\psi_1:= \kappa\theta_2\psi_0(1-\sigma_0)^{-1}$ gives $\|P_Ne_2\|\leq \psi_1\|e_0\|$. 
 Since we enforce $|\gamma_k|\leq 1$ for all $k$ with safeguarding, it follows that 
 \begin{align}\label{main-range-bound}
     \|P_Re_2\|\leq c\max\set{\|e_1\|^2,\|e_0\|^2}.
 \end{align}
 By standard Newton theory we have $\|e_1\|<\|e_0\|$, so in this case $\|P_Re_2\|<c\|e_0\|^2$. Equations \eqref{main-null-bound} and \eqref{main-range-bound} give
$\|e_2\| \leq \|P_Re_2\|+\|P_Ne_2\|\leq \big(c\|e_0\|+\psi_1\big)\|e_0\|.$
Reducing $\|e_0\|$ and $\sigma_0$ if necessary, we can ensure $\eta_1:=\big(c_1\|e_0\|+\psi_1\big)<1$. Hence  
$\|e_2\|<\eta_1\|e_0\|$.
Next we consider $\sigma_2$. Applying the reverse triangle inequality directly to $L_{2}^{\lambda}(e_1+w_2,e_0+w_1)$ gives $\|P_Ne_2\|\geq c(1-\nu_{2})\max\{|1-\gamma_{2}|\,\|P_Ne_1\|,|\gamma_2|\,\|P_Ne_0\|\}$ 
It follows that
\begin{align}
    \sigma_2\leq \frac{c\max\set{|1-\gamma_{2}|\,\|e_1\|^2,|\gamma_{2}|\,\|e_{0}\|^2}}{(1-\nu_{2})\max\set{|1-\gamma_{2}|\,\|P_Ne_1\|,|\gamma_2|\,\|P_Ne_0\|}}\leq \frac{c(1+\sigma_0)}{1-\nu_{2}}\|e_0\|,
\end{align} 
The last inequality follows since $\|e_1\|<\eta_0 \|e_0\|<\|e_0\|$. By \cref{gamma-safeguarding}, $c(1+\sigma_0)/(1-\nu_{2})<c_6$. Once more reducing $\|e_0\|$ if necessary, we can ensure 
\begin{align}\label{main-sigma-bound}
    \sigma_2<c_6\|e_0\|<s_0\sigma_0. 
\end{align}
Hence $W_2\subset W_0$, and $x_3$ is well-defined. Thus \eqref{main-null-bound}, \eqref{main-range-bound}, and \eqref{main-sigma-bound} all hold with each subscript increased by one since $C_1(\sigma_{1,2},\|e_{1,2}\|)=c\max\set{\sigma_1,\sigma_2,\|e_1\|,\|e_2\|}<c\max\set{\sigma_0,\sigma_1,\|e_0\|,\|e_1\|}=C_1(\sigma_{0,1},\|e_{0,1}\|)$. 
If we let $\psi_2:=\kappa\theta_3\psi_1$, then $\|e_3\|<(c\max\set{\eta_0,\eta_1}^2\|e_0\|+\psi_2)\|e_0\|$. Let $\eta_2:=(c\max\set{\eta_0,\eta_1}^2\|e_0\|+\psi_2)$ so that $\|e_3\|<\eta_2\|e_0\|$. As $\max\set{\eta_0,\eta_1}\leq \hat{\eta}<1$, $\eta_2\leq \max\set{\hat{\eta},\kappa\theta_3}(c\|e_0\|+\psi_1)=\max\set{\hat{\eta},\kappa\theta_3}\eta_1<\eta_1$ since $\kappa\theta_3<1$. Let $\phi=\max\set{\hat{\eta},\kappa}<1$ so that $\eta_2<\phi\eta_1<\hat{\eta}$. Therefore $\|e_3\|<\|e_0\|$ and  $W_3\subset W_0$. 
Proceeding inductively, it can be shown that for all that for all $k\geq 1$, $W_{k+1}\subset W_0$, $\|P_Ne_{k+1}\|< \kappa\theta_{k+1}||Pe_k||<\psi_{k}\|e_0\|$, and $\|e_{k+1}\|<\eta_k\|e_0\|$. This proves that $x_k$ remains well defined, and establishes \eqref{range-bd-main-app} and \eqref{null-rate-of-convergence-main-app}. 

It remains to show that $e_k$ goes to zero. It suffices to prove that $\set{\eta_k}$ defined by $\eta_k=c\max\set{\eta_{k-1},\eta_{k-2}}^2\|e_0\|+\psi_k$ converges to zero. Again by induction, one can show that $\eta_k$ decreases, and for $k\geq 7$ we have $\eta_{k+1}<\phi^k\eta_1$. Thus $\eta_k\to 0$, and therefore $\|e_{k+1}\|\to 0$. This completes the proof.
\end{proof}

Note that the essential parts of the proof rely on the structure of $\|P_Ne_k\|$ and $\|P_Re_k\|$. For any order root, the error expansions are of the form 
\begin{align}
	P_N(e_k+w_{k+1}) &= \delta P_Ne_k+\hat{T}_kP_Re_k+\bigo(\|e_k\|^2)\\
	P_R(e_k+w_{k+1}) &= \bigo(\|e_k\|^{\mu})
\end{align}
where $\delta\in(0,1)$ and $\mu\geq 2$. Thus an analogous argument can be used to prove \ref{thm:convergence-ho}. The only noticable difference occurs at equation \eqref{main-null-bound}, where in the general case $\kappa\in(d/(d+1),1)$. This, however, does not effect convergence since $\theta_{k+1}\leq 1$ for all $k$. 

\end{document}